\documentclass[english]{amsart}

\usepackage[all,cmtip]{xy}
\usepackage{hyperref}
\usepackage{blindtext}

\usepackage{geometry}
\usepackage{amsmath,amssymb,amscd,amsfonts}

\usepackage{babel}
\usepackage{amstext}
\usepackage{amsmath}
\usepackage{amsfonts}
\usepackage{latexsym}
\usepackage{ifthen}

\usepackage[all,cmtip]{xy}
\xyoption{all}
\usepackage{enumitem}
\setlist[enumerate]{label=(\thethm.\arabic*), before={\setcounter{enumi}{\value{equation}}}, after={\setcounter{equation}{\value{enumi}}}}

\newcommand{\CC}{\mathbb{C}}
\newcommand{\Q}{\mathbb{Q}}
\newcommand{\DD}{\mathbb D}
\newcommand{\Z}{\mathbb{Z}}

\newcommand{\ddbar}{\partial\bar{\partial}}

\newcommand{\wh}{\widehat}
\newcommand{\wt}{\widetilde}
\newcommand{\cX}{\mathcal{X}}
\newcommand{\cF}{\mathcal{F}}
\newcommand{\cE}{\mathcal{E}}

\newcommand{\Lie}{\mathcal{L}}
\newcommand{\cC}{\mathcal{C}}

\newcommand{\cP}{\mathcal{P}}

\renewcommand{\O}{\mathcal{O}}

\newcommand{\ep}{\varepsilon}
\renewcommand{\epsilon}{\varepsilon}

\newcommand{\ol}{\overline}

\renewcommand{\leq}{\leqslant}
\renewcommand{\geq}{\geqslant}

\newcommand{\Div}{\mathrm{Div}}

\newcommand{\D}{D}

\newcommand{\Ker}{\mathrm{Ker}}

\newcommand{\dbar}{\bar \partial}

\newtheorem{thm}{Theorem}[section]
\newtheorem{lemme}[thm]{Lemma}
\newtheorem{proposition}[thm]{Proposition}
\newtheorem{quest}[thm]{Question}
\newtheorem{claim}[thm]{Claim}

\newtheorem{conjecture}[thm]{Conjecture}

\newtheorem{cor}[thm]{Corollary}
\newtheorem{construction}[thm]{Construction}
\newtheorem{notation}[thm]{Notation}
\newtheorem{remark}[thm]{Remark}

\newtheorem{hyp}[thm]{Hypothesis}

\numberwithin{equation}{thm}

\title{Infinitesimal extension of pluricanonical forms}
\author{Junyan CAO, Mihai PAUN}

\address{Université Côte d’Azur, CNRS, LJAD, France}
\email{junyan.cao@unice.fr}

\address{Universit\"at Bayreuth, Mathematisches Institut, Lehrstuhl Mathematik VIII, Universit\"atsstrasse 30, D-95447, Bayreuth, Germany}
\email{mihai.paun@uni-bayreuth.de}

\begin{document} 

\maketitle

%\begin{dedication}
%\vspace{-1.5cm}
%{To Bo Berndtsson,\\ On the Occasion of His $70^{\rm th}$ Birthday}
%\end{dedication} 
%\vskip 1cm

\section{Introduction}

\noindent Let $p:\cX\to \DD$ be a holomorphic family of smooth, $n$-dimensional compact manifolds whose central fiber -denoted by $X$- is assumed to be K\"ahler. We denote by $\displaystyle K_\cX$ the canonical bundle of $\cX$ and let $\Lie\to \cX$ be an arbitrary line bundle.
\smallskip

\noindent We are interested in the extension properties of global sections of the adjoint bundle $K_\cX+ \Lie$ defined over the $k^{\rm th}$ infinitesimal neighborhood of the central fiber of $p$, where $k\geq 1$ is a positive integer.   
In other words, we consider the sheaves 
\begin{equation}\label{intr1}
\cF_k:= \left(K_\cX+ \Lie\right)\otimes \O_\cX/t^{k+1}\O_\cX,
\end{equation} 
and the corresponding spaces of sections $H^0(\cX, \cF_k)$.
\medskip

\noindent In this article we are establishing sufficient criteria for a given $\displaystyle s\in H^0\left(\cX, \cF_k\right)$ to be in the image of the map induced by the projection
\begin{equation}\label{intr2}\pi_k: \cF_{k+1}\to \cF_k.
\end{equation}
We are particularly interested in the case 
$\Lie= (m-1)K_\cX$ where $m\geq 1$ is a positive integer,
so we state next our main result in this setting. In order to do so we have to introduce further notations. Let $s$ be a holomorphic section of $\cF_k$; we denote by 
$h_L$ the metric on the bundle $L:= (m-1)K_X$ induced by the restriction of $s$ to the central fiber $X$.
%\begin{equation}\label{intr4}
%\mathfrak I\subset \mathcal O_X, \qquad \mathfrak I= \lim_{\ep\to 0}\mathcal I\left((1-\ep)\frac{m-1}{m}\Sigma\right)
%\end{equation}
%where for any $\mathbb R$-divisor $\displaystyle \Delta$ we denote by $\mathcal I(\Delta)$ the associated multiplier ideal sheaf.
\medskip

\noindent The following statement is the main result of the first part of our paper.

\begin{thm}\label{MT} Let $s$ be a section of the sheaf \eqref{intr1} with $\Lie= (m-1)K_\cX$, which admits  
a $\mathcal C^\infty$ extension $s_k$ so that if we write
%\begin{equation}\label{not2}
$\displaystyle \dbar s_k= t^{k+1}\Lambda_k,$
%\end{equation}
the integral 
\begin{equation}\label{intr7}
\int_X\left|\frac{\Lambda_k}{dt}\right|^2e^{-(1-\ep)\varphi_L}dV< \infty
\end{equation} 
converges for any positive $\ep> 0$.
Then there exists a section $\wh s$ of $\cF_{k+1}$ such that $s= \pi_k(\wh s)$.
\end{thm}
\noindent The notations in \eqref{intr7} are as follows: we measure the $(n,1)$-form $\displaystyle \frac{\Lambda_k}{dt}\Big|_X$ with respect to an arbitrary Kähler metric on $X$ and the metric $e^{-(1-\ep)\varphi_L- \ep \phi_L}$ where $e^{-\phi_L}$ is a smooth metric on $L$ and $e^{-\varphi_L}=h_L$ is the singular metric induced by $s$.

\medskip

\noindent Coming back to the general setting, let $\cE_1,\dots, \cE_K$ be a set of line bundles 
 on $\cX$ such that 
\begin{equation}\label{intr5}
(1- r_0)c_1(\Lie)= r_0c_1(K_\cX)+ \sum_{i=1}^K r_ic_1(\cE_i),
\end{equation}
where $0\leq r_0< 1$ and  $r_j\geq 0$ for $j=2,\dots, K$ are rational numbers.
In \eqref{intr5} we denote by $c_1(\cE)$ the first Chern class of $\cE$. 
Let $s$ be a section of $\cF_k$ and and for each $i= 1,\dots K$ let $\sigma_i$ be a section of the sheaf $\displaystyle \cE_i\otimes \O_\cX/t^{k+2}\O_\cX$ (notice that this is one order higher than $s$). We denote by $h_L$ the metric induced on the bundle $L:= \Lie|_X$ by the set of holomorphic sections $(s, \sigma_i)$ restricted to the central fiber
--notice that this makes sense thanks to \eqref{intr5}. 
\medskip

\noindent Then we have the following version of Theorem \ref{MT}. 
\begin{thm}\label{007} Let $\displaystyle (s, \sigma_i)_{i=1,\dots, K}$ be a family of sections of $\displaystyle (\cF_k, \cE_i)_{i=1,\dots, K}$ respectively as above. We assume that $s$ admits a $\mathcal C^\infty$ extension $s_k$ such that if we write
$\displaystyle \dbar s_k= t^{k+1}\Lambda_k$
then \begin{equation}\label{intr6}\int_X\left|\frac{\Lambda_k}{dt}\right|^2e^{-(1-\ep)\varphi_L}dV< \infty\end{equation} for any positive $\ep> 0$. Then there exists a section $\wh s$ of $\cF_{k+1}$ such that $s= \pi_k(\wh s)$.
\end{thm}
\medskip

\noindent A very interesting result in connection to Theorem \ref{007} was established by Cao, Demailly and Matsumura in \cite{CDM}, as follows.

\begin{thm}\label{CDM}\cite{CDM} Consider $p: {\mathcal X}\to \Delta$ a K\"ahler family, $(\mathcal L, h_{\mathcal L})$ a line bundle on $\mathcal X$ 
and let $u\in H^0(\mathcal X, \mathcal F_k)$ be a holomorphic section such that: 

\begin{itemize}

\item We have $\displaystyle i\Theta(\mathcal L, h_{\mathcal L})\geq 0$ on $\mathcal X$. 
    \smallskip 

 \item $u$ admits a $\mathcal C^\infty$ extension $u_k$ to $\mathcal X$ such that $\displaystyle 
\overline\partial u_k= t^{k+1}\Lambda_k$ together with
\[ \int_{\mathcal X}\left|{\Lambda_k}\right|^2e^{-\varphi_L}dV< \infty
\] 
 \end{itemize}  
Then $u$ extends to $\mathcal X$ (no estimates available).
\end{thm}

\noindent We remark that a-priori in Theorem \ref{MT} we do not have any semi-positively curved metric on the bundle $\mathcal L= (m-1)K_\cX$. 
Moreover, the $L^2$ condition in Theorem \ref{CDM} shows that our hypothesis \eqref{intr7} is natural, at least from the point of view of extension theorems.  
\medskip

\noindent Theorem \ref{007} is equally motivated by an important conjecture in K\"ahler geometry that we next recall. After Y.-T. Siu's "invariance of plurigenera" articles \cite{Siu98, Siu00} (see also \cite{Pau07}) concerning the extension of $s$ in case of a \emph{projective family} $\cX$,
the following very important problem is still open (despite of the \emph{avalanche} of articles and crucial achievements that followed Siu's work).

\begin{conjecture}\label{Siuconj}\cite{Siu00}
  Let $p:\cX\to \DD$ be a family of smooth, $n$-dimensional compact manifolds whose central fiber is denoted by $X$. We assume that the total space $\cX$ admits a K\"ahler metric. Then any holomorphic pluricanonical section defined on $X$ extends holomorphically to $\cX$.
\end{conjecture}
\smallskip

\noindent  Among the articles dedicated to Conjecture \ref{Siuconj} we first mention here \cite{Lev83} and \cite{Lev85}, due to M.~Levine. A direct consequence of Theorem \ref{MT} is the following statement, which represents a more general version of the results in \cite{Lev83, Lev85}.

\begin{cor} \label{corlev}
Let $p:\cX\to \DD$ be a smooth holomorphic family whose central fiber $X$ is K\"ahler, and let $\displaystyle s\in H^0(X, mK_{X})$
be a pluricanonical section defined on $X$. We assume that the set of zeros of the 
ideal $$\mathfrak I:= \lim_{\ep\to 0}\mathcal I\left((1-\ep)\frac{m-1}{m}\Sigma\right)$$ is discrete, 
where $\Sigma$ is the divisor corresponding to $s$.
 Then any section $\tau \in H^0 (X, m K_X)$ admits a holomorphic extension to a neighborhood of the center fiber $X$.
 If moreover the total space $\cX$ is Kähler, then any section $\tau \in H^0 (X, m K_X)$ admits a holomorphic extension to $\cX$.
\end{cor}
\smallskip

\noindent  Furthermore we obtain a new proof of the next statement, which is the main result in \cite{Lev83} (cf. Section 6). Our main objective in doing that is to understand better the set-up in which the hypothesis \eqref{intr6} is verified.
\begin{thm}\label{Lev83} \cite{Lev83}
Let $\displaystyle s\in H^0(X, mK_{X})$ be a pluricanonical section defined on the 
central fiber. We assume that the set of zeroes $(s=0)\subset X$ of $s$ is non-singular.
Then $s$ admits an extension to $\cX$.
\end{thm}

\noindent We highlight next some parts of the proof of Theorem \ref{MT} which seem interesting to us. By hypothesis the section $s$ admits a $\mathcal C^\infty$ extension $s_k$ such that
\begin{equation}\label{int100}
\dbar s_k= t^{k+1}\Lambda_k
\end{equation}
for some form $\Lambda_k\in \mathcal C^\infty_{n+1, 1}(\cX, (m-1)K_{\cX})$ whose restriction to $X$ is $\dbar$-closed. The first step is to show that there exist forms $\alpha_k$ and $\beta_k$
which are smooth in the complement of the support of $(s=0)\subset X$ and such that
\begin{equation}\label{int101}
\frac{\Lambda_k}{dt}\Big|_X= \dbar \alpha_k+ D'(\beta_k)
\end{equation}
holds pointwise in the complement of $s=0$, where $D'$ is the $(1,0)$-part of the Chern connection on $L:= (m-1)K_X$
induced by the given section $\displaystyle s= s_k|_X$.
\smallskip

\noindent Note that the equality \eqref{int101} is true in general, i.e. without any $L^2$ assumption. But the problem is that the $L$-valued forms $\alpha_k$ and $\beta_k$ have singularities of type $\displaystyle \frac{\Phi}{s^{k+1}}$, where $\Phi$ is smooth. We therefore find ourselves in a rather strange situation, in which the LHS of \eqref{int101} is non-singular but the primitives appearing in
the RHS are meromorphic. Of course, this is -so to speak- the price to pay in order to have an intrinsic decomposition as in \eqref{int101}.
\medskip

\noindent Anyway, in order to prove Theorem \ref{MT}, it would be sufficient to establish the following result which we state separately since we find it interesting in itself.

\begin{thm}\label{007'} Let $\gamma$ be a $\dbar$-closed smooth $(n, 1)$-form on a compact K\"ahler manifold $X$ with values in $L$ such that the following hold.
	\begin{enumerate}
		
		\item[\rm (a)] The bundle $L$ is endowed with a singular metric $h_L=e^{-\varphi_L}$ such that
		$$\varphi_L= \sum_i  q_i\log|f_i|^2 +C^\infty, $$
		for some rational coefficients $q_i \in \mathbb Q^+$, where $f_i$ are holomorphic functions and such that the curvature current $i\Theta_{h_L}$ is supported in a proper subvariety of $X$.
		
		\item[\rm (b)] We suppose that 
		$$\int_X |\gamma|^2 e^{-(1-\ep) \varphi_L} <+\infty $$
		for every $\ep >0$. 
		
		\item[\rm (c)]  Set $Y:=\{ h_L = +\infty\}$.
		There exists a $(n-1,1)$-form $\beta$ and a $(n, 0)$-form $\alpha$ with poles along $Y$ such that 
		\begin{equation}\nonumber
			\gamma= \dbar\alpha+ D' _{h_L}(\beta)
		\end{equation}  
		on $X\setminus Y$, where $D' _{h_L} $ is induced by the singular metric $h_L$ on $L$. 
	\end{enumerate}
	Then $\gamma$ is $\dbar$-exact.
\end{thm}
\medskip

\noindent The assumption (c) above means that the coefficients of the two forms can be locally written as 
$\displaystyle \frac{g}{\prod f_i^N}$, where $g$ is smooth. For the proof of Theorem \ref{007'}, an important point is that after the log-resolution of $\big(X, \Div(s)\big)$ combined with a few other reductions the singularities of $\alpha$ and $\beta$ can be improved significantly. 
\medskip

\noindent The last step of the proof consists in showing that Theorem \ref{007'} is a consequence of the following general vanishing theorem, established in \cite{CP23}.

\begin{thm}\label{ddbar} \cite{CP23}
	Let $X$ be a $n$-dimensional compact Kähler manifold and let $E$ be a snc divisor on $X$, $s_E$ be the canonical section of $E$.   Let $(L, h_L)$ be a holomorphic line bundle on $X$ such that
		$$i\Theta_{h_L} (L) = \sum q_i [Y_i] + \theta_L,$$ where $q_i \in ]0,1[ \cap \mathbb Q$, $E +\sum Y_i$ is snc, and the form $\theta_L$ is smooth, semi-positive.
Let $\lambda$ be a $\dbar$-closed smooth $(n,q)$-form with value in $L+E$. If there exists $\beta_1$ and $\beta_2$, two $L$-valued $(n-1, q)$-form and $(n-1, q-1)$-form with log poles along $E+\sum Y_i$ respectively, such that
\begin{equation}\label{yes!}
\frac{\lambda}{s_E} = D' _{h_L}\beta_1 + \theta_L \wedge \beta_2 \qquad\text{ on } X\setminus (E+\sum Y_i) ,
\end{equation}	
then the form $\lambda$ is $\dbar$-exact, i.e, 
	the class $[\lambda]=0 \in H^q (X, K_X +L+E)$.
\end{thm}
\medskip

\noindent In the second part of our article we are concerned more specifically with (extensions of) sections of $\cF_i$ for $i=0,1$. The majority of our results can be seen as applications of Theorem \ref{MT}, as follows. 
\smallskip

\noindent $\bullet$  In subsection \ref{smoothdiv},  we provide an alternative argument for Theorem \ref{Lev83}. The main 
point is to show that if the divisor $s=0$ is smooth, then for each order $k$ we can construct a particular vector field 
$\Xi$, so that the resulting forms $\alpha_k$ and $\beta_k$ are automatically smooth (hence in particular $L^2$, given that the multiplier ideal sheaf corresponding to $\frac{m-1}{m}\Div(s)$ is trivial). Then we can extend $s$ over infinitesimal nbd's of any order and the conclusion follows. 
\smallskip

\noindent $\bullet$ Even if the hypothesis \eqref{intr7} in statement \ref{MT} fits perfectly in the $L^2$-landscape, it 
is somehow unnatural in the context of Siu's conjecture. 
Nevertheless in subsection \ref{sufcond}, we formulate  
a sufficient condition for \eqref{intr7} to hold, as follows. Consider a pluricanonical section $s$ on the central fiber of $p:\cX\to \DD$, such that the next requirement is satisfied.
\begin{hyp}\nonumber \label{cond}
There exists a divisor $\Sigma$ on $\cX$ such that we have. 
\begin{enumerate}
\smallskip

\item[(i)] The divisor $\Sigma+ X$ is snc.
\smallskip

\item[(ii)] The support of $\Div(s)$ is contained in the restriction $\displaystyle \Sigma|_X$.
\end{enumerate}
\nonumber
\end{hyp}
\medskip

\noindent In Section 5, among many other results we show that the following holds true.
\begin{thm}\label{hypo} Let $s\in H^0(X, mK_X)$ be a pluricanonical section such that Hypothesis \ref{cond} is satisfied. Then for each $k\geq 0$ 
the section $s$ admits a $\mathcal C^\infty$ extension $s_k$ to $\cX$, such that \eqref{intr7} holds true. In particular, $s$ admits a holomorphic extension to $\cX$.
\end{thm}

\begin{remark}{\rm
Given a pluricanonical section $s$ defined on the central fiber $X$ of the family $p$, we can always find a birational map $\pi:\wh \cX\to \cX$
such that the support of the inverse image of the set $(s=0)$ is a divisor, transverse to the proper transform $\ol X$ of $X$. The map $\wh p:= p\circ \pi: \wh \cX\to \mathbb D$ 
is now singular, but the support of the central fiber is snc. The said support is the union of $\ol X$ with $\pi$-exceptional divisors,
and the inverse image of $s$ becomes a pluricanonical section of $\ol X$, whose set of zeros is contained in a divisor defined globally.
In other words, the context is vaguely similar to that in Hypothesis \ref{cond}, but the new map $\wh p$ is singular, so the methods used in proving Theorem \ref{hypo} do not apply (at least, not directly).  } 
\end{remark}
\smallskip

\noindent $\bullet$ In subsection \ref{2nd}, we discuss the second order extension. The problem extending sections of the sheaf $\cF_1$ is already very interesting, since we have to deal with an unreduced space.  The partial result we obtain in this direction is the following.

\begin{thm}\label{order2}
Let $\displaystyle s\in H^0(X, mK_{X})$ be a pluricanonical section defined on the 
central fiber $X$ of a family $p:\cX\to \D$ of compact K\"ahler manifolds. We assume that there exists a non-singular vector field $\Xi$ on the total space $\cX$ such that 
\begin{equation}\label{stupid_cond1}
  dp(\Xi)= \frac{\partial}{\partial t}, \qquad \sup_X
  \frac{\big|\dbar \Xi |_X\big|^2_{\omega_E}}{\log^2\frac{1}{|s|^2}}\leq C,
\end{equation}
where $\omega_E$ is a metric with Poincar\'e singularities along the set
$(s= 0)$.
Then there exists a section $s_2\in \mathcal C^\infty(\cX, mK_{\cX})$ such that 
$\displaystyle s_2|_X= s$ and \begin{equation}\label{int1}
\dbar s_2= t^3\Lambda_2,\end{equation} where $\Lambda_2$ belongs to the space $\mathcal C^\infty_{0,1}(\cX, mK_\cX)$.
\end{thm}

\medskip

\noindent We discuss next the main technical result needed in the proof of Theorem \ref{order2}. The general setup is as follows.
Let $(L, h_L)$ be a Hermitian line bundle on $X$, endowed with a metric with analytic singularities. We assume that the curvature current
$\Theta(L, h_L)\geq 0$ is semi-positive. 
Let $v$ be an $L$-valued form of type
$(n-1, 1)$, such that $D'v$ is $\dbar$-closed. We are interested in the $\dbar$-equation 
\begin{equation}\label{int2}
\dbar u= D'v.
\end{equation}
\noindent We assume further that $X$ is endowed with a Poincaré-type metric $\omega_E$,
with poles along the singular locus of $h_L$. The result we obtain is as follows.
\begin{thm}\label{L2I}
  Assume moreover that $v$ is $L^2$ with respect to a Poincar\'e-type metric $\omega_E$ on $X$. Then the equation \eqref{int2} admits a solution $u$ such that
\begin{equation}\label{int2111}
\int_X|u|^2e^{-\varphi_L}\leq \int_X|v|^2_{\omega_E}e^{-\varphi_L}dV_{\omega_E}.
\end{equation}    
\end{thm}
\noindent The new aspect of Theorem \ref{L2I} is that we do not assume 
that $D'v$ belongs to $L^2$.
\smallskip

\noindent It is a very interesting question to decide whether the equation \eqref{int2} can be solved with estimates involving an incomplete metric $\omega$ on $X\setminus (h_L= \infty)$. Let $\omega$ be non-singular K\"ahler metric on $X$.
In this direction we obtain the following result.

\begin{thm}\label{L2II} Let $(L, h_L)$ be a Hermitian line bundle, such that $\Theta(L, h_L)\geq 0$ is semi-positive. Let $v$ be an $L$-valued form of type
$(n-1, 1)$, such that $D'v$ is $\dbar$-closed. We assume that the following hypothesis are satisfied.
\begin{enumerate}
\smallskip

\item[\rm (1)] The metric $h_L$ has analytic singularities, and let $Z$ be the support of the set $(h_L= \infty)$.

\item[\rm (2)] The integrals
\begin{equation}\nonumber
  \int_X|v|^2_\omega e^{-\varphi_L}dV_\omega,\qquad  \int_X
  |D' v|^2 e^{-\varphi_L}dV_{\omega}
\end{equation}
are convergent.
\smallskip

\item[\rm (3)] There exists a complete metric $\omega_Z$ on $X\setminus Z$ such that we have 
\begin{equation}\nonumber
  \int_X|v|^2_{\omega_Z} e^{-\varphi_L}dV_{\omega_Z}< \infty
  \end{equation}

\end{enumerate}

\noindent Then the equation \eqref{int2}
has a solution $u$ such that
\begin{equation}\label{int4}
\int_X|u|^2e^{-\varphi_L}\leq \int_X|v|^2_{\omega}e^{-\varphi_L}dV_{\omega}.
\end{equation}
\end{thm}
\medskip

\begin{remark}
For applications, it would be very important to remove the hypothesis (3) in Theorem \ref{L2II}. However, we are not sure whether 
the statement is still correct.
\end{remark}

\medskip

\noindent Other than the references mentioned above, the articles \cite {Nog95},
\cite{LRW} by J. Noguchi, K. Liu, S. Rao and X.~Wan contain important ideas in connection with our work here. Recent and exciting contributions in the direction of Conjecture 1.1 are due to
J.-P. Demailly in \cite{Dem20} as well as to S.~Rao and I.-H. Tsai in \cite{RT}.

\medskip

\noindent{\bf Acknowledgments.} During part of the preparation of this article we have enjoyed the hospitality and excellent working conditions offered by the \emph{Freiburg Institut for Advanced Study}. At FRIAS we had the opportunity to meet many visitors, but the math exchanges we have had with S.~Kebekus and C. Schnell have been exceptionally fruitful. We thank A. H\"oring, J. Lott, R. Mazzeo, S. Rao and Y.~Rubinstein for valuable discussions. JC thanks the Institut Universitaire de France and the A.N.R JCJC project Karmapolis (ANR-21-CE40-0010) for providing excellent working condition. Finally, our work was completed during MP's visit at \emph{Center for Complex Geometry} (in Daejeon, South Korea). Many thanks to J.-M. Hwang for the invitation and vibrant working atmosphere in this institute!

%And this superb paper was written by us, Junyan CAO and Mihai PAUN.
\medskip

\noindent This paper is organized as follows.
\tableofcontents

%%%%%%%%%%%%%%%%%%%%%%%%%%%%%%%%%%%%%%%%%%%%%%%%%%%%%%%%%%%%%%%%%%%%%%%%%%%%%%%%%%%%%%%%%%%%%%%%%%%%%%%%%%%%%%%%%%%%%%%%%%%%%%%%%%%%%%%%%%%%%%%%%%

\section{First order differential operators}\label{comm}

In this subsection we will recall a few facts from differential geometry of line
bundles. We take this opportunity to fix some notations as well.

\subsection{Connection induced by a smooth section}

\begin{notation}\label{notation:1}
  In the setting of our article, let $\Lie\to \cX$ be a holomorphic line
  bundle endowed with a connection whose $(0,1)$-component is given by the
  $\dbar$ operator,
  \begin{equation}\label{eq73}
    \nabla= D^\prime_{\cX}+ \dbar.
  \end{equation}  
  Since no confusion is likely, we will use the same symbol to denote the
  induced operator on the space of smooth, $\Lie$-valued $(p,q)$-forms,
  \begin{equation}\label{eq74}
    D^\prime_{\cX}: \cC^\infty\left(\cX, \Omega_{\cX}^{p, q}\otimes \Lie\right)\to \cC^\infty\left(\cX, \Omega_{\cX}^{p+1, q}\otimes \Lie\right).
  \end{equation}
 \end{notation}

\begin{construction}\label{cons:1}
 Let $\Lie\to \cX$ be a holomorphic line bundle, and let
  $\wt s$ be a smooth section of $\Lie$, with vanishing locus $Z \subseteq \cX$.
  Assume that the open cover $(U_i)_i$ trivializes the bundle $\Lie$ and choose
  trivialization, $L|_{U_i} \cong \mathcal{O}_{U_i}$.  The section $\wt s$ will
  therefore give rise to smooth functions $\wt s_i$ on $U_i$.  Set
  $U^\circ_i := U_i \setminus Z$.

  Given any index $i$ we can now define a differential operator
  $D'_i : \mathcal{O}_{U^\circ_i} \to \Omega^1_{U^\circ_i}$ on $U^\circ_i$ as
  follows,
  \begin{equation}\label{eq76}
    D'_i : \mathcal{O}_{U^\circ_i} \to \Omega^1_{U^\circ_i}, \qquad %
    \sigma \mapsto  D^\prime_t(\sigma_i)dt+ \sum_\alpha D^\prime_\alpha(\sigma_i)dz_i^\alpha
  \end{equation}  
  where
  \begin{align*}\label{eq77}
    D^\prime_t(\sigma_i) & := \partial_t \sigma_i- \frac{\partial_t \wt s_i}{\wt s_i}\sigma_i, & D_\alpha^{\prime}(\sigma_i) & :=  \partial_\alpha \sigma_i- \frac{\partial_\alpha \wt s_i}{\wt s_i}\sigma_i. \\
    \partial_t & := \partial/\partial t & \partial_\alpha & := \partial/\partial z_i^\alpha
  \end{align*} 
  Using the trivialization chosen above, we can view $D'_i$ as differential
  operators
  $D'_i : \Lie|_{U^\circ_i} \to \Lie|_{U^\circ_i} \otimes \Omega^1_{U^\circ_i}$ on
  $U^\circ_i$.  One computes without much pain that these differential operators
  glue to give a globally defined operator
  $$
  D'_{\cX} : \Lie \to \Lie \otimes \Omega^1_{\cX}
  $$
  on $\cX \setminus Z$.  In particular, we obtain a connection
  $\nabla := D'_{\cX} + \dbar$ on $\Lie|_{\cX \setminus Z}$.
\end{construction}

\begin{lemme}
  Setting as in Construction~\ref{cons:1}.  Then, we have the following graded
  commutator identity
  \begin{equation}\label{eq78}
    [D^\prime_\cX, \dbar]= -\dbar\left(\frac{\partial \wt s}{\wt s}\right)\wedge 
  \end{equation}  
 \end{lemme}
\begin{proof}
  Direct computation that we skip. We note that the symbol $\displaystyle 
  \dbar\left(\frac{\partial \wt s}{\wt s}\right)$ which appears in \eqref{eq78}
  is a global $(1,1)$ form on the complement of the set $\wt s= 0$, as we now explain. The global section $\wt s$ corresponds to local smooth functions denoted by $\wt s_i$ such that we have
  \begin{equation}\label{exp_1}
\wt s_i= g_{ij}\wt s_j
  \end{equation}
  on overlapping subsets of $\cX$. Then we have
\begin{equation}\label{exp_2}
\frac{\partial \wt s_i}{s_i}= \frac{\partial g_{ij}}{g_{ij}}+ \frac{\partial \wt s_j}{s_j}
  \end{equation}  
and since $\displaystyle \dbar\left(\frac{\partial g_{ij}}{g_{ij}}\right)= 0$, we obtain
\begin{equation}\label{exp_3}
\dbar \left(\frac{\partial \wt s_i}{s_i}\right)= \dbar\left(\frac{\partial \wt s_j}{s_j}\right)\end{equation} 
which proves our claim. We note that if $\wt s$ is holomorphic, then
the $(1,1)$-from is simply zero on $\cX \setminus Z$.
\end{proof}
\smallskip

\begin{remark} We note that in general, the differential operator $\nabla$ defined in \eqref{eq73}
does not coincides with the Chern connection induced by the metric corresponding to 
$\wt s$. This is of course the case if $\wt s$ is holomorphic.
\end{remark}
\smallskip

\begin{remark} 
As in the usual case, a smooth section $\wt s$ of the bundle $mL$ induces a connection on $L$ itself by a slight modification of the construction above that is to say, by multiplication with a constant, cf. next subsection.
\end{remark}

\subsection{Lie derivative and commutation relations}\label{komm}
Let
\begin{equation}\label{eq16}
  s\in H^0(X, m K_X)
\end{equation}  
be a holomorphic section of the pluricanonical bundle, where $m\geq 1$ is a positive integer. 

Let $\wt s$ be arbitrary smooth extension of the section $s$. As already hinted, we can use the section $\wt s$ in order to
define a connection on the bundle
$\Lie:= (m-1) K_{\cX}$ restricted to the complement of the set of zeros of 
$\wt s$. The local differential operators corresponding to $D^\prime_\cX$
are given by
\begin{equation}\label{eq200}
  D^\prime_t(\sigma):= \partial_t \sigma- \frac{m-1}{m}
    \frac{\partial_t \wt s_i}{\wt s_i}\sigma, \quad D_\alpha^{\prime}(\sigma):=
    \partial_\alpha \sigma- \frac{m-1}{m}
  \frac{\partial_\alpha \wt s_i}{\wt s_i}\sigma.
\end{equation} 
where $\partial_t:= \partial/\partial t$ and 
$\partial_\alpha:= \partial/\partial z_i^\alpha$
for each $\alpha=1, \dots, n$. In \eqref{eq200} the symbol $\sigma$ represents a local section of the bundle $\displaystyle (m-1)K_{\cX/\D}$. The sum
$$D^\prime_t(\sigma)dt+ \sum_\alpha D_\alpha^{\prime}(\sigma)dz_i^\alpha$$ corresponds to a global connection of $(1, 0)$ type on $L:= (m-1)K_{\cX}$. If $\wt s$ would have been holomorphic, then this is nothing but the Chern connection; in any case, this is well (i.e. globally) defined. It is clear that we have 
\begin{equation}
D' _\cX \circ D' _\cX =0 . 
\end{equation}

We consider next a smooth vector field $\displaystyle \Xi$ on the total space $\cX$ 
such that $dp  (\Xi) = \frac{\partial}{\partial t}$ on $\DD$.
It can be written in local coordinates as follows
\begin{equation}\label{eq18}
  \Xi|_{U_i}= \frac{\partial}{\partial t}+ \sum_{\alpha=1}^n
  v_i^\alpha\frac{\partial}{\partial z_i^\alpha}
\end{equation}  
where $v_i^\alpha$ for $i=1,\dots, n$ are smooth functions. 
\medskip

\noindent Our vector field induces a Lie derivative operator $\Lie_\Xi$ as follows
\begin{equation}\label{eq19}
\Lie_\Xi(\sigma):= D'_\cX(\Xi\rfloor \sigma)
\end{equation}
where $\sigma$ is any $(n+1, q)$-form with values in $L$. The result $\Lie_\Xi(\sigma)$
is a form of the same type as $\sigma$, and it will play an important role in what follows.
\medskip

\noindent Let $i_\Xi$ be the operator of degree (-1, 0) given by the contraction with the vector field $\Xi$. Since we have $\Lie_{\Xi}= [D' _{\cX} , i_{\Xi}]$ on $(n+1, q)$-differential forms, the following Jacobi identity 
$$
[\dbar, \Lie_{\Xi}] +[D' _{\cX}, [i_{\Xi}, \dbar]] + [i_\Xi, [\dbar, D' _{\cX}]]= 0$$
holds true, and therefore we obtain the next formula over $(n+1, q)$-forms
\begin{equation}\label{eq20}
[\dbar, \Lie_{\Xi}]= -D' _{\cX} \circ  (\dbar \Xi)\rfloor  -  [i_\Xi, [\dbar, D' _{\cX}]].
\end{equation}
\medskip

\noindent Our next statement is playing an important role in the ``algebra" part of the proof of our main results. 

\begin{lemme}\label{Komm} Let $\rho$ be a $(n-1, 1)$-form with values in $L$ on $\cX$. Then we have the equality 
\begin{equation}\label{eq21}
\Lie_\Xi \big(D' _\cX (dt\wedge \rho)\big) = D' _\cX \left( dt\wedge \Big( \Xi  \rfloor D' _\cX \big(\Xi \rfloor (dt \wedge \rho)\big) \Big)\right)
\end{equation} 
on $\cX$.
\end{lemme}
\begin{proof}
The argument is quite clear; we start with the left hand side and we remark that we have 
\begin{equation}\label{eq22}
D' _\cX (dt\wedge \rho)= -dt\wedge D' _\cX (\rho)
\end{equation}
and so
\begin{equation}\label{eq23}
\Xi\rfloor \big(D' _\cX (dt\wedge \rho)\big)= - D' _\cX (\rho)+ dt\wedge\big(\Xi\rfloor D' _\cX (\rho)\big)
\end{equation}
since the contraction with the vector field $\Xi$ is a derivation. The LHS of \eqref{eq21} is therefore equal
to
\begin{equation}\label{eq24}
-dt\wedge D' _\cX\big(\Xi\rfloor D' _\cX (\rho)\big).
\end{equation}

\noindent For the RHS, we have 
\begin{equation}\label{eq25}
\Xi \rfloor (dt \wedge \rho)= \rho- dt \wedge (\Xi \rfloor \rho)
\end{equation}
so
\begin{equation}\label{eq26}
D' _\cX \big(\Xi \rfloor (dt \wedge \rho)\big)= D' _\cX\rho+ dt \wedge D' _\cX(\Xi \rfloor \rho).
\end{equation}
A contraction with $\Xi$ gives
\begin{equation}\label{eq27}
\Xi\rfloor \Big(D' _\cX \big(\Xi \rfloor (dt \wedge \rho)\big)\Big)\equiv \Xi\rfloor \big(D' _\cX\rho\big)+ D' _\cX(\Xi \rfloor \rho)
\end{equation}
modulo a term in $dt\wedge \cdot$, so we have 
\begin{equation}\label{eq28}
dt\wedge \left(\Xi\rfloor \Big(D' _\cX \big(\Xi \rfloor (dt \wedge \rho)\big)\Big)\right)= dt\wedge \Big(\Xi\rfloor \big(D' _\cX\rho\big)\Big)+ dt\wedge D' _\cX(\Xi \rfloor \rho)
\end{equation}
and a further derivative of \eqref{eq28} shows that the LHS of \eqref{eq21} equals
\begin{equation}\label{eq29}
-dt\wedge D' _\cX\big(\Xi\rfloor D' _\cX (\rho)\big),
\end{equation}
so our statement is proved. In the argument just finished, we have used many times the fact that $D' _\cX\circ D' _\cX= 0$.
\end{proof}

%%%%%%%%%%%%%%%%%%%%%%%%%%%%%%%%%%%%%%%%%%%%%%%%%%%%%%%%%%%%%%%%%%%%%%%%%%%%%%%%%%%%%%%%%%%%%%%%%%%%%%%%%%%%%%%%%%%%%%%%%%%%%%%%%%%%%%%%%%%%%%%%%%%%%%%%%%%%%%%%%%%%%%%%%%%%%%%%%%%%%
\medskip

\section{Proof of Theorem \ref{MT}, Theorem \ref{007'} and Corollary \ref{corlev}}

\noindent We have divided our arguments for Theorem \ref{MT} in several steps. The first one consists in showing that the restriction
\begin{equation}\label{new66}
  \gamma:= \frac{\Lambda_k}{dt}\Big|_X\end{equation}
belongs to the image of $\dbar+ D'$, provided that we admit forms with meromorphic coefficients, cf. section \ref{eval}. This will show that the problem we have to solve is equivalent to Theorem \ref{007'}. The $L^2$ hypothesis \eqref{intr7} is used in sections \ref{red}
and \ref{cos} in order to reduce Theorem \ref{007'} to the vanishing statement in \cite{CP23}.

\subsection{Evaluation of the obstruction form}\label{eval} 
\noindent Let $\displaystyle \bigcup U_i$ be a Stein cover of $\cX$. We consider a section $\displaystyle s\in H^0 (X, K_\cX+ \Lie|_X)$ on the central fiber of $p$. We denote by $s_i$ an arbitrary holomorphic extension of
$\displaystyle s|_{U_i\cap X}$. Then we have 
\begin{equation}\label{eq34}
	s_idz_i\wedge dt\otimes e_i=
	s_jdz_j\wedge dt\otimes e_j
	+ t\Lambda_{ij}dz_j\wedge dt\otimes e_j. 
\end{equation}
on overlapping coordinate sets $U_i\cap U_j$. Hence we interpret $s$ as a top form on $\cX$ with values in $\displaystyle \mathcal L$. In \eqref{eq34} the symbol $e_i$ stands for the local frame of $\displaystyle \mathcal L$ (which, in our case
is simply $\displaystyle (dz_i\wedge dt)^{\otimes(m-1)}$). 
\smallskip

\noindent We can reformulate this data as follows: \emph{there exists a smooth section $s_{1}$ of the bundle $K_{\cX}+ \Lie$ such that}
\begin{equation}\label{eq35}
	\dbar s_{1}= t\Lambda_{1}
\end{equation}
\emph{on $\cX$, and whose restriction to the central fiber is equal to $s$.}
\medskip

\noindent Consider next the Lie derivative operator $\Lie_\Xi$ associated to a vector field $\Xi$ such that $\displaystyle dp(\Xi)= \frac{\partial}{\partial t}$. If we apply $\Lie_\Xi$ to \eqref{eq35}
on the RHS we obtain
\begin{equation}\label{eq36}
	\Lambda_{1}+ t\Lie_\Xi(\Lambda_{1})
\end{equation}
By the commutation relation \eqref{eq20}, the LHS of \eqref{eq35} becomes
\begin{equation}\label{eq37}
	\dbar\big(\Lie_\Xi(s_1)\big)+ D'_\cX\big(\dbar\Xi\rfloor s_1\big)+ [i_\Xi, [\dbar, D' _{\cX}]] s_1
\end{equation}
Modulo a factor divisible with $t$, the last term of \eqref{eq37} is equal to
$\displaystyle \frac{m-1}{m}\Lambda_1$ on $X \setminus Z$. Therefore we get
\begin{equation}\label{eq38}
	\frac{1}{m}\Lambda_1\equiv \dbar\big(\Lie_\Xi(s_1)\big)+ D'_\cX\big(\dbar\Xi\rfloor s_1\big) 
\end{equation}
on $X\setminus Z$, 
modulo a form which is divisible by $t$ and non-singular on $X\setminus Z$.
\medskip

\noindent Another interesting observation is that the form $\dbar\Xi\rfloor s_1$ can be written as follows
\begin{equation}\label{eq39}
	\dbar\Xi\rfloor s_1= dt\wedge \rho_1
\end{equation}
given the shape of the vector field $\Xi$, cf. \eqref{eq18}. We have therefore obtained the 
first step of the next statement.

\begin{lemme}\label{linalg}
	Let $s_k$ be a smooth section of the bundle $K_{\cX}+ \Lie$, such that 
	\begin{equation}\label{eq339}
		\dbar s_k= t^{k+1}\Lambda_k\end{equation}
	for some $(n+1, 1)$-form $\Lambda_k$. We assume that the 
	connection on $L= (m-1)K_\cX$ is induced by the section $s_k$. Then we can find the forms $\alpha_k$ and $\beta_k$ 
	such that we have 
	\begin{equation}\label{eq40}
		c_k\Lambda_k\equiv \dbar \alpha_k+ D'_\cX\big(dt\wedge \beta_k) \qquad\text{on }  X \setminus Z
	\end{equation} 
	modulo the ideal generated by $t$. Moreover, the forms $\alpha_k$ and $\beta_k$ are 
	smooth in the complement of the set $s_k= 0$ and $c_k$ is a positive constant.
\end{lemme}

\begin{proof}
	This is obtained as follows: we take the Lie derivative in \eqref{eq40} and use the commutation relation \eqref{eq20}. The result of this first operation is that we have 
	\begin{equation}\label{eq41}
		\frac{k+1}{m}t^{k}\Lambda_k\equiv \dbar \big(\Lie_\Xi(s_k)\big)+ 
		D'_\cX\big(\dbar\Xi\rfloor s_k\big) \qquad\text{on }  X \setminus Z
	\end{equation}
	modulo the ideal generated by $t^{k+1}$. In order to start the inductive process which will prove our statement, we rewrite \eqref{eq41} as follows
	\begin{equation}\label{eq50}
		\frac{k+1}{m}t^{k}\Lambda_k\equiv \dbar u_1+ 
		D'_\cX\big(dt\wedge v_1\big) \qquad\text{on }  X \setminus Z,
	\end{equation}
	modulo $t^{k+1}$. Here we use the notations 
	\begin{equation}\label{eq51}
		u_1:=  \Lie_\Xi(s_k), \qquad dt\wedge v_1:= \dbar\Xi\rfloor s_k.
	\end{equation}
	We show next that if we apply the operator $\Lie_\Xi$ to the RHS of \eqref{eq51}
	the result is an expression of a similar type. Indeed, we have 
	\begin{equation}\label{eq52}
		\Lie_\Xi(\dbar u_1)= \dbar\left(\Lie_\Xi(u_1)\right)+ D'_\cX\big(\dbar\Xi\rfloor u_1\big)
		\qquad\text{on }  X \setminus Z
	\end{equation}
	modulo the curvature term which is of order $t^k$: this is one higher than we have to take into account, so we just drop it.
	
	\noindent Also, we have 
	\begin{equation}\label{eq53}
		\Lie_\Xi \big(D' _\cX (dt\wedge v_1)\big) = D' _\cX \left( dt\wedge \Big( \Xi  \rfloor D' _\cX \big(\Xi \rfloor (dt \wedge v_1)\big) \Big)\right)
	\end{equation} 
	thanks to Lemma \ref{Komm}, so summing up we get
	\begin{equation}\label{eq54}
		\frac{k(k+1)}{m}t^{k-1}\Lambda_k\equiv \dbar u_2+ 
		D'_\cX\big(dt\wedge v_2\big),
	\end{equation}
	modulo $t^{k}$, where we use the notations
	\begin{equation}\label{eq55}
		u_2:= \Lie_\Xi(u_1), \qquad dt\wedge v_2:=  \dbar\Xi\rfloor u_1+ dt\wedge \Big(\Xi  \rfloor D' _\cX \big(\Xi \rfloor (dt \wedge v_1)\big) \Big)
	\end{equation}
	Our statement follows by induction on $k$ --and moreover we have
	\begin{equation}
		c_k:= \frac{(k+1)!}{m}
	\end{equation}
	as we see by taking successive derivatives of \eqref{eq54}.
\end{proof}
\medskip

\begin{remark}
	The relations \eqref{eq55} give the explicit process of constructing $\alpha_k$ and $\beta_k$.
	They will play an important role in the analysis of the singularities of these forms.  
      \end{remark}

      \medskip

\begin{remark}
Lemma \ref{linalg} shows that in order to conclude, it would be sufficient to prove Theorem \ref{007'}. 
\end{remark}

%%%%%%%%%%%%%%%%%%%%%%%%%%%%%%%%%%%%%%%%%%%%%%%%%%%%%%%%%
%%%%%%%%%%%%%%%%%%%%%%%%%%%%%%%%%%%%%%%%%%%%%%%%%%%%%%%%%%%%%%%%
%%%%%%%%%%%%%%%%%%%%%%%%%%%%%%%%%%%%%%%%%%%%%%%%%%%%%%%%

\subsection{The $L^2$ hypothesis}\label{red}
Let $\pi:\wh X\to X$ be a birational map such that
$$\pi^* \varphi_L = \sum m_i \log |s_{E_i}| + \sum_{j\in J}q_j \log |s_{F_j}| + C^\infty ,$$
holds true where the convention is that $m_i\in \Z_+$, $q_j\in \Q_+\setminus \Z$,  $s_{E_i}, s_{F_j}$ are the canonical sections associated to the divisors $E_i, F_j$, and $\sum E_i + \sum F_j$ is simple normal crossing.
We can write thus
\begin{equation}\label{new2}
	\wh L := \pi^\star L = \sum_{i\in I} m_i E_i+ \sum_{j\in J}q_jF_j + L_0, 
\end{equation}
where $L_0$ is a $\mathbb Q$-line bundle equipped with a smooth metric $h_0$ such that $i\Theta_{h_0} (L_0)= 0$.
\smallskip

We consider the form $\gamma$ in \eqref{new66}. Its inverse image $\pi^\star \gamma$ is a $(n, 1)$-form with values in $\wh L$, such that 
\begin{equation}\label{new1}
\int_X\left|\pi^\star \gamma \right|^2_{\wh \omega}e^{-(1-\ep)\varphi_{\wh L}}dV_{\wh \omega}< + \infty
\end{equation}
holds true for every positive $\ep$, where $\wh \omega$ is any (non-singular) metric on $\wh X$.

\noindent
As consequence of \eqref{new1}, we claim that the form 
\begin{equation}\label{new3}
\lambda:= \frac{\pi^\star \gamma}{\prod_{i\in I}s_{E_i}^{m_i-1}\prod_{j\in J}s_{F_j}^{\lfloor q_j\rfloor}}
\end{equation}
is smooth -and of course, $\dbar$-closed. Indeed, let $(z_1,\dots, z_n)$ be a coordinate system defined on some open subset $\Omega\subset \wh X$ of $\wh X$. We write
$$\pi^\star \gamma|_{\Omega}= \sum a_{\ol i}dz\wedge d\ol z_i$$
and then \eqref{new1} is equivalent to the following relations
$$\int_\Omega\left|a_{\ol i} \right|^2e^{-(1-\ep)\varphi_{\wh L}}dV< + \infty$$
for all indexes $i$, so our claim follows. 
\medskip 

\noindent The equality
\begin{equation}\label{new100}
\gamma= \dbar\alpha_k+ D'\beta_k
\end{equation}  
established in Lemma \ref{linalg} becomes  
\begin{equation}\label{new4}
\lambda= \dbar \alpha+ D'\beta
\end{equation}
on $\wh X\setminus (E\cup F)$, where $D'$ is the $(1, 0)$-part of the Chern connection for 
singular metric on the line bundle $\O\left(E+ \sum (q_j- \lfloor q_j\rfloor)F_j + L_0 \right)$, and $\alpha, \beta$ are some forms with value in $\O\left(E+ \sum (q_j- \lfloor q_j\rfloor)F_j + L_0 \right)$ and have some poles along $E+F$.
\smallskip

\noindent Thus, by changing the notations we have the following setup.

\begin{enumerate}
\smallskip

\item[(1)] $\lambda$ is a $\dbar$-closed form of type $(n, 1)$ and values in $E+ L$ on a compact K\"ahler manifold $X$.
\smallskip

\item[(2)] $E+ F$ is a snc divisor on $X$ and $L= \O\left(\sum (q_j- \lfloor q_j\rfloor)F_j +L_0\right)$
is a line bundle which we endow with the singular metric $h_L$ induced by $F$ and $L_0$.
\smallskip

\item[(3)] $\alpha$ and $\beta$ are two $E+ L$-forms of type $(n, 0)$ and $(n-1, 1)$, respectively
with poles at most along $E+F$, i.e.
\begin{equation}\label{twoforms}
\alpha = \frac{1}{s^{N}}\O(1), \qquad \beta= \frac{1}{s^N}\O(1)
\end{equation}
where $s= s_E\cdot s_F$ and such that \eqref{new4} holds in the complement of 
the support of $E+F$.
\end{enumerate}

\noindent Our next step would be to improve the "shape" of the two forms in \eqref{twoforms}: we show in the next subsection that we can replace them with forms having logarithmic poles on $E+F$.
%%%%%%%%%%%%%%%%%%%%%%%%%%%%%%%%%%%%%%%%%%%%%%%%%%%%%%%%%%%%%%%%%%%%%%%%%%%%%%%%%%%%%%%%%%%%%%%%%%%%%%

\subsection{Improving the pole order of $\alpha$ and $\beta$ (cosmetics)}\label{cos}
Notice that by hypothesis (3) above, the meromorphic form 
$\displaystyle \frac{\lambda}{s_E}$ has logarithmic poles along $E$; it can therefore be interpreted as a $(0,1)$-form with values in $\Lambda^n T^\star_X\langle E\rangle\otimes L$. 
In this set-up, we formulate the next statement.

\begin{proposition}\label{c1} There exists $\alpha_1$ and $\beta_1$ differential forms with log-poles along $E$, such that 
\begin{equation}\label{log1}
\dbar \alpha_1+ D'\beta_1= \frac{\lambda}{s_E}
\end{equation} 
holds in the complement of $E$, where $D'$ is the $(1,0)$-part of the Chern connection for $(L, h_L)$.
 \end{proposition}

\begin{proof} We start by choosing a finite covering $\displaystyle (\Omega_i)_{i\in I}$ of $X$, together with corresponding coordinates $(z_i^1,\dots, z_i^n)$ such that for every index $i\in I$ we have
\begin{equation}\label{pave90}
E_j\cap \Omega_i= (z_i^j=0), \qquad F_k\cap \Omega_i= (z_i^k=0)
\end{equation}
for each $j= 1,\dots, r$ and $k=r+1,\dots, r+s$.
\smallskip
	
\noindent Consider next the component $E_1$ of $E$. We then define an associated vector field 
\begin{equation}\label{log-20}
V_1 := \sum \theta_i(z)z_i^1\frac{\partial}{\partial z_i^1}
\end{equation}
where $\displaystyle (\theta_i)_{i\in I}$ is a partition of unit and $\Omega_i\cap E_1= (z_i^1= 0).$
Notice that since $E_1$ is a global hypersurface, we have 
\begin{equation}\label{log-21}
V_1 |_{\Omega_i}= \big(z_i^1+ \O (z_i^1)^2\big)\frac{\partial}{\partial z_i^1}+ 
\sum_{k= 2}^{r+s}\O(z_i^1z_i^k)\frac{\partial}{\partial z_i^k}+ \sum_{k> r+s}\O(z_i^1)\frac{\partial}{\partial z_i^k}
\end{equation}
where the notation $\O(z_i^1)$ means a function belonging to the ideal 
$\O(-E_1)\otimes\cC^\infty(\Omega_i)$. Indeed, given that the transition functions for the coordinate system we have are
\begin{equation}
z_i^\alpha= g_{ij}^{\alpha}(z_j) z_j^\alpha
\end{equation}
the expression \eqref{log-21} obviously corresponds to the restriction of $V_1$ to $\Omega_i$.
\smallskip

\noindent Notice also that the local expression of the covariant derivative $D'$ induced by the singular metric on $E+L$ reads as
\begin{equation}\label{pave91}
D'|_{\Omega_i}= \partial - \sum_{j=1}^r\frac{dz_i}{z_j}- \sum_{k=r+1}^{r+s}q_k\frac{dz_k}{z_k}.\end{equation}
	
\medskip

\noindent We first modify the form $\alpha$.  
Assume that $\alpha$ has a pole of order $N$ along $E_1$ for some $N\geq 1$.   
We define	\begin{equation}\label{log4}
	\alpha_N :=  \alpha   +\frac{1}{N} D'\left(V_1 \rfloor\alpha\right) 
	\end{equation}
where $D'$ is the $(1,0)$ part of the Chern connection corresponding to the singular metric on $E+ L$,
cf. \eqref{pave91}. By a direct computation, we show now that the pole order of $\alpha_N$ is at most $N-1$. Indeed, we write $\displaystyle \alpha|_{\Omega}= \frac{f}{z_1^N}dz$ (where $\Omega$ is one of the $\Omega_i$; we also drop the index $i$) and then we have
\begin{equation}\label{pave92}
V_1\rfloor \alpha= \left(z_1+ \O (z_1)^2 \right)\frac{f}{z_1^N}\wh{dz_1}+ \sum_{j\geq 2}(-1)^{j-1}\O (z_1)
\frac{f}{z_1^N}\wh{dz_j}.
\end{equation}
By the formula \eqref{pave91} we have 
\begin{equation}\label{pave93}
D'\left(V_1\rfloor \alpha\right)\equiv -N\frac{f}{z_1^N}dz
\end{equation}
modulo a form whose pole order is smaller than $N-1$, so our claim is proved
\smallskip

\noindent Moreover, we have 
	\begin{equation}\label{log5}
	\lambda = \dbar\alpha_N+ D' (\beta +\frac{1}{N} \dbar \left(V_1 \rfloor\alpha\right) )  
\end{equation}
in the complement of $E\cup F$. We repeat this procedure until the exponent $N$ is equal to zero. At this moment we have  
	\begin{equation}\label{log6}
	\lambda= \dbar \alpha_0 + D'(\beta_0)
	\end{equation}
	where $\alpha_0$ has no poles along $E_1$ and the coefficients of $\beta_0$ are similar to the ones of our initial $\beta$.
\medskip

\noindent We are dealing with the other components of $E$ and $F$ similarly.  For example for $E_2$ we proceed as follows.  We define the vector field
\begin{equation}\label{comp1}
V_2:= \sum \theta_i(z)z_i^2\frac{\partial}{\partial z_i^2}
\end{equation}
constructed by using coordinate system \eqref{pave90}. It then follows that locally the $V_2$ can be written as 
\begin{equation}\label{comp3}
V_2 |_{\Omega_i}= \big(z_i^2+ \O(z_i^2)^2\big)\frac{\partial}{\partial z_i^2}+ \O(z_i^1z_i^2)\frac{\partial}{\partial z_i^1}+ \sum_{j\neq 1,2} \O(z_i^2)\frac{\partial}{\partial z_i^k}.
\end{equation}
With this choice of vector field, if $\alpha$ is a smooth form along $E_1$ then the same is true for 
\begin{equation}\label{addb}
\alpha + \frac{1}{N}D' \left(V_2\rfloor \alpha\right)
\end{equation}
in other words, we are not loosing what we have gained at the first step.
Furthermore,  the pole order of \eqref{addb} along $E_2$ has dropped by at least one unit if $\alpha$ is of pole order $N\geq 1$ along $E_2$. The same procedure as before applies. 
\smallskip

\noindent Finally we get some new $\alpha$ and $\beta$ such that 
	\begin{equation}\label{log7}
	\lambda = \dbar{\alpha}+ D'{\beta}
      \end{equation}
holds in the complement of $E+F$ and such that the form $\displaystyle {\alpha}$ is smooth on $X$ and the coefficients of $\beta$ are of type $\displaystyle \frac{b}{s^k}$ for some smooth function $b$ and positive integer $k$.  
	%\smallskip
	
	%\noindent Summing up, we have 
	%\begin{equation}\label{log-2}
	%\frac{\lambda}{s_E}= \dbar \alpha + \partial\gamma
	%\end{equation}
	%where the coefficients of $\gamma$ are of type $\displaystyle \frac{b}{z_1^k}$ for some smooth function $b$ and positive integer $k$.  
	\medskip
	
	\noindent We now turn to the form $\beta$, and prove that modulo the image of the $D'$-operator it has no poles along $E+F$. 
	Assume that $\beta$ has a pole of order $N\geq 1$ along $E_1$. Then we define   
	\begin{equation}\label{log888}\beta_N :=\beta +\frac{1}{N} D' (V_1 \rfloor \beta)\end{equation}
and the same type of calculations as before will show that it has a pole of order at most $N-1$ along $E_1$. 

\noindent To see this, we write  
	\begin{equation}\label{log8}
	\beta|_{\Omega}= \sum_{j\geq 2, k} \frac{\eta_{j \ol k}}{z_1^{N}}dz_1\wedge \wh{dz_j}\wedge d\ol z_k
	+ \sum_{k} \frac{\eta_{\ol k}}{z_1^N}\wh{dz_1}\wedge d\ol z_k
	\end{equation}
in local coordinates.	
Since $\lambda$ and $\alpha$ in \eqref{log7} are smooth, $D'\beta$ has no poles along $E_1$. Therefore for each $k$, $\displaystyle \eta_{\ol k}$ is divisible by $z_1$. We can then write
	\begin{equation}\label{log-28}
	\beta |_{\Omega}= \sum_{j\geq 2, k} \frac{\eta_{j \ol k}}{z_1^{N}}dz_1\wedge \wh{dz_j}\wedge d\ol z_k
	+ \sum_{k} \frac{\eta_{\ol k}}{z_1^{N-1}}\wh{dz_1}\wedge d\ol z_k.
	\end{equation}
        By using the expression \eqref{log-21}, a direct computation -which we skip- shows that adding $\displaystyle \frac{1}{N}D' (V_1 \rfloor \beta)$ to our form $\beta$ has the effect of removing the first sum in \eqref{log-28}, and so $\beta_N$
has a pole of order at most $N-1$ along $E_1$. We still have
\begin{equation}\label{log71}
\lambda = \dbar{\alpha}+ D'{\beta_N} \qquad\text{on } X\setminus (E+F) ,
\end{equation} 
where $\alpha$ is unchanged. We repeat this procedure until the exponent $N$ is equal to zero. We get 
\begin{equation}\label{log66}
\lambda= \dbar \alpha + D'(\beta)
\end{equation}
where $\alpha$ is smooth on $X$ and (the resulting) $\beta$ has no pole along $E_1$. 
\smallskip

\noindent A little observation is in order at this point: if we write the form $\beta$ in coordinates
\begin{equation}\label{log-288}
\beta |_{\Omega}= \sum_{j\geq 2, k} {\eta_{j \ol k}}dz_1\wedge \wh{dz_j}\wedge d\ol z_k
	+ \sum_{k} {\eta_{1\ol k}}\wh{dz_1}\wedge d\ol z_k,
	\end{equation}
where $\displaystyle \eta_{j\ol k}$ and $\displaystyle \eta_{1\ol k}$ are smooth with respect to $z_1$, the equality \eqref{log66} shows that $D'\beta$ is smooth. By the explicit expression of the operator $D'$ is \eqref{pave91} it follows that $\displaystyle \eta_{1\ol k}$ is divisible by $z_1$, for each $k$.
\smallskip	
	
\noindent In order to treat the remaining components of $E$ and $F$ we are using the argument already employed for $\alpha$,
 namely we consider $\beta + \frac{1}{N} D' (V_2 \rfloor \beta)$. This form has no pole along $E_1$ (this is where the divisibility of the coefficient $\displaystyle \eta_{1\ol k}$ by $z_1$ is playing an important role) and pole order along $E_2$ drops. 
\medskip

\noindent Finally we obtain two smooth forms $\alpha$ and $\beta$ such that 
\begin{equation}\label{log666}
\lambda= \dbar \alpha + D'(\beta)
\end{equation}
holds at each point of $X\setminus (E\cup F)$. If we write 
\begin{equation}\label{log6667}
\beta= \sum_{j, k} {\eta_{j \ol k}}\wh{dz_j}\wedge d\ol z_k
\end{equation}
then we automatically have $\displaystyle \eta_{j\ol k}$ divisible by $z_j$, for all $j=1,\dots, r+s$. In particular \eqref{log666} holds point-wise on $X$ and by division with the section $s_E$ we have 
\begin{equation}\label{log6668}
\frac{\lambda}{s_E}= \dbar \alpha_1 + D'(\beta_1)
\end{equation}
on $X\setminus E$, where $\alpha_1$ and $\beta_1$ have logarithmic poles along $E$.
The proposition is proved.
\end{proof}

\begin{remark}\label{preciselog}
Assume that $E$ is equal to zero. Then the particular shape of the coefficients of $\beta= \beta_1$
in \eqref{log6667} shows that 
\begin{equation}
\int_X|\beta|_{\omega_\cP}^2e^{-\varphi_L}< \infty 
\end{equation}
in other words, $\beta$ is $L^2$ with respect to the Poincar\'e metric. 
\end{remark}		
\medskip

%%%%%%%%%%%%%%%%%%%%%%%%%%%%%%%%%%%%%%%%%%%%%%%%%%%%%%%%%%
%%%%%%%%%%%%%%%%%%%%%%%%%%%%%%%%%%%%%%%%%%%%%%%%%%%%%%%%%%
%%%%%%%%%%%%%%%%%%%%%%%%%%%%%%%%%%%%%%%%%%%%%%%%%%%%%%%%%%
\subsection{End of the Proof}

\begin{proof}[Proof of Theorem \ref{MT}]
Thanks to Lemma \ref{linalg}, we can find $\alpha_k$ and $\beta_k$ which are smooth in the complement of $s_k=0$, and of poles along $s_k =0$ such that
	\begin{equation}\label{equll}
	\frac{\Lambda_k}{dt}= \dbar \alpha_k+ D'\beta_k 
\end{equation} 
where $D'$ is the connection with respect to the singular metric induced by $s_k |_X$.
By Proposition \ref{c1} and Theorem \ref{ddbar},  it follows that 
$\lambda$ is $\dbar$-exact. Given the definition of $\lambda$, we infer that 
$\displaystyle \frac{\Lambda_k}{dt}$ is $\dbar$-exact, i.e. $\displaystyle \frac{\Lambda_k}{dt} = \dbar \gamma_k$ for some smooth $(n,0)$-form $\gamma_k$ with values in $L := (m-1)K_\cX$.
\smallskip

\noindent Consider next a Stein cover $U_i$ of $\cX$. Since $\Lambda_k$ is $\dbar$-closed, on every open set $U_i$, we can find a $L$-valued $(n+1,0)$-form $\Gamma_{k,i}$ such that 
$$ \frac{\Gamma_{k,i}}{dt}|_{X \cap U_i} =  \gamma_k \qquad\text{and} \qquad   \Lambda_k = \dbar \Gamma_{k,i} \text{ on } U_i.$$
Indeed, we can construct $\Gamma_{k,i}$ as follows: start with an arbitrary $L$-valued $(n+1,0)$-form $\Psi_{k, i}$ on $U_i$ such that 
$\displaystyle \dbar \Psi_{k, i}= \Lambda_k|_{U_i}$. This is of course possible, since $\Lambda_k$ is $\dbar$-closed. It follows that the difference
$$\frac{\Psi_{k,i}}{dt}|_{X \cap U_i}- \gamma_k:= \tau_{k,i}$$
is holomorphic on the intersection $X\cap U_i$, and let $\wt \tau_{k,i}$ be an arbitrary holomorphic extension of $\tau_{k,i}$ to $U_i$. Then we
define $\Gamma_{k,i}:= \Psi_{k,i}- \wt \tau_{k,i}$.
\smallskip 

\noindent In particular $\Gamma_{k,i} - \Gamma_{k,j}$ is holomorphic and it equals zero when restricted to the central fiber, in other words, 
divisible by $t$. 
Now we define $\Gamma_k :=\sum_i \theta_i \Gamma_{k,i}$, where $\theta_i$ is a partition of unity for the covering $U_i$.
Thus we have 
\begin{equation}\label{debar1}
{\Lambda_k}\simeq \dbar \Gamma_k 
\end{equation}
modulo the ideal generated by $t$. We denote by 
\begin{equation}\label{debar2}
s_{k+1}:= s_k- t^{k+1}\Gamma_k 
\end{equation}
and then we have 
\begin{equation}\label{debar3}
\dbar s_{k+1}= t^{k+2}\Lambda_{k+1}
\end{equation}
where $\Lambda_{k+1}$ is a smooth $(n+1, 1)$-forms with values in $(m-1)K_\cX$. 
Theorem \ref{MT} is thus proved.	
\end{proof}

\begin{proof}[Proof of Corollary \ref{corlev}]
Let $\tau \in H^0 (X, m K_X)$ be an arbitrary section. Let $a \in \mathbb R^+$ and let $\Sigma_a$ be the divisor corresponding to $s+ a \tau$. Then for $a$ small enough, the ideal 
$$\mathcal{J}_a := \lim_{\ep \to 0} \mathcal{I} ((1-\ep) \frac{m-1}{m} \Sigma_a )$$
is discrete.  To extend $\tau$, it is sufficient to extend $s+ a \tau$ for any $a$ small enough.

We suppose by induction that there exists a smooth extension $s_k$ of $s+ a \tau$ such that $\dbar s_k = t^{k+1} \Lambda_k$.
To prove the corollary, it is sufficient to prove that $\Lambda_k$ is $\dbar$-exact. 
Let 
\begin{equation}\label{vers1} 
\lambda_k:= \frac{\Lambda_k}{dt}\Big|_X 
\end{equation}
be the restriction of $\Lambda_k$ to $X$. Then we have $\dbar \lambda_k=0$. 
Since $\mathfrak J_a$ is discrete,  the image of $\lambda_k$  via the projection morphism
\begin{equation}\label{vers2} 
H^1(X, mK_X)\to H^1\big(X, \O_X (mK_X)\otimes \O_X/\mathfrak J\big)
\end{equation}
is equal to zero. It follows that there exist a $\mathcal C^\infty$ form $a_k$ of type $(n,0)$ with values in $(m-1)K_X$, as well as a $\mathcal C^\infty$ form $b_k$ of type $(n, 1)$ with values the sheaf $K_X^{\otimes(m-1)}\otimes \mathfrak J$ so that 
\begin{equation}\label{vers3} 
\lambda_k= b_k+ \dbar a_k.
\end{equation}
Thanks to Theorem \ref{MT}, we know that $b_k$ is $\dbar$-exact. Then $\lambda_k$ is $\dbar$-exact. It means that we can find a smooth extension $s_{k+1}$ such that $\dbar s_{k+1} = t^{k+1} \Lambda_{k+1}$. By induction, $s+a \sigma$ admits an extension for any order. Therefore $s+a \tau$ admits a holomorphic extension to a neighborhood of the center fiber $X$ cf. \cite[Lemma 1.1]{Lev83}.

\medskip

Now we assume that the total space $\cX$ is Kähler. We consider the $m$-relative Bergman kernel  metric $h_B$ on $K_{\cX /\Delta}$ introduced in \cite{BP08}. We know that $i\Theta_{h_B} ( K_{\cX /\Delta}) \geq 0$ on $\cX$.  Thanks to the first part of the corollary, $\tau$ admits a holomorphic extension to a neighborhood of the center fiber $X$, i.e., $S \in H^0 (p^{-1}(U), m K_{\cX/ \Delta})$ and $S |_X =\tau$, where $U$ is a neighborhood of $0\in \Delta$. 
Then there is a constant $C$ such that for any $t$ close to $0$, we have 
$$\int_{\cX_t} |S |_{\cX_t}| ^{\frac{2}{m}}  \leq C .$$
Then we have the pointwise estimate $|S|_{\cX_t}  |^{\frac{2}{m}}  _{h_B}  \leq C$ for a generic $t$ close to $0$.  As a consequence, $\int_X |\tau|_{h_B} ^2 <+\infty$. By using Ohsawa-Takegoshi extension, $\tau$ admits a holomorphic extension to the total space $\cX$. 
\end{proof}	
\medskip

\noindent The proof of Theorem \ref{007} follows along the same lines as in Theorem \ref{MT}, so we only sketch it very briefly.
\begin{proof}
Notice first that since $\sigma_i$ is  a section of $\mathcal{E}_i\otimes \mathcal{O}_{\mathcal{X}} / t^{k+2} \mathcal{O}_{\mathcal{X}}$,
the curvature term containing $\sigma_i$ vanishes on $X\setminus Div (\sigma_i)$ up to order $k$. 
Then the same argument as Lemma \ref{linalg} implies that
\begin{equation}\label{varthm}
\lambda_k = \dbar \alpha_k + D' (\beta_k) \qquad\text{on } X \setminus (Div(s) + \sum_i Div (\sigma_i)).
\end{equation}
The scheme of the proof of Theorem \ref{MT} applies \emph{mutatis mutandis} and we infer that $\lambda_k$ is $\dbar$-exact. The theorem is proved.
\end{proof}

%%%%%%%%%%%%%%%%%%%%%%%%%%%%%%%%%%%%%%%%%%%%%%%%%%%%%%%%%%%%%%%%%%%%%%%%%%%%%%%%%%%%%%%%%%%%%%%%%%%%%%%%%%%%%%%%%%%%%%%%%%%%%%%%%%%%%%%%%%%%%%%%%%%%%%%%%

\section{More results and applications}\label{ana}

%%%%%%%%%%%%%%%%%%%%%%%%%%%%%%%%%%%%%%%%%%%%%%%%%%%%%%%%%%%%%%%%%%%%%%%%%%%%%%%%%%%%%%%%%%%%%%%%%%%%%%%%%%%%%%%%%%%%%%%%%%%%%%%%%%%%%%%%%%%%%%%%%%%%%%%%%

\subsection{On first order extension}\label{firstorderex} 
The first result that we establish here is that one can always 
extend a pluricanonical section $s$ defined on $X$ to the first infinitesimal neighborhood. This is  established in the paper \cite{Lev83} by different methods. 

\noindent Let $\displaystyle \big(\Omega_i, (t,z_i)\big)_{i\in I}$ be an arbitrary coordinate system on $\cX$. This induces automatically a trivialization of the relative canonical bundle, so our section $s$ corresponds to a family of holomorphic functions $\displaystyle (f_i)_{i\in I}$.
We obtain here the following more precise version of \cite{Lev83}, as follows.
\begin{thm} Let $s$ be a section of the pluricanonical bundle $mK_X$ defined on the central fiber of the map $p:\cX\to \mathbb D$. Then $s$ 
admits a smooth extension $s_1$ such that 
\begin{equation}\label{july224}
s_1|_{\Omega_i}\simeq \left(f_i(z_i)+ t f_{i1}(z_i)\right)(dt\wedge dz_i)^{\otimes m}
\end{equation}
modulo $\O(|t|^2)$, where $f_{i1}$ are holomorphic such that $\displaystyle \int_{X\cap \Omega_i}\frac{|f_{i1}|^2}{|f_i|^{2\frac{m-1}{m}}}d\lambda(z_i)< \infty$. 
\end{thm}

\noindent In other words, the restriction of the Lie derivative $\Lie_\Xi(s_1)$ of $s_1$ to the central fiber is $L^2$ with respect to the 
weight induced by $s$, for any vector field $\Xi$.
\medskip

\begin{proof}
Let $s_0$ be any smooth extension of $s$, so that we have 
\begin{equation}\label{new11}
	\dbar s_0= t\Lambda_0
\end{equation}
and define
 $\displaystyle \lambda_0 :=\frac{\Lambda_0}{dt} |_X$ be the restriction of $\Lambda_0$ to the central fiber.
 
 \noindent Let $\Xi$ be any vector field as in \eqref{eq18}. 
As we have seen in the section \ref{eval}, by applying the Lie derivative $\Lie_\Xi$ to \eqref{new11} we have
\begin{equation}\label{eq42}
\frac{1}{m}\lambda_1= \dbar a_1 + D'\big(b_1\big)
\end{equation}
on $X\setminus Z$, where we use the (abuse of...) notation
\begin{equation}\label{eq43}
a_1:= \frac{\Lie_\Xi(s_0)}{dt}\Big|_X,\qquad b_1:= \rho_1|_X
\end{equation}
and $\rho_1$ is the form already appearing in Lemma \ref{linalg}.
\medskip

\noindent The form $D'\big(b_1\big)$ is certainly closed on $X\setminus Z$. In order to apply 
Theorem \ref{007'}, we have to verify the $L^2$ hypothesis. We recall that here the data is
\begin{equation}\label{eq44}
L:= (m-1)K_X, \qquad \varphi_L:= \frac{m-1}{m}\log |s_i|^2
\end{equation}
where $s_i$ is the local expression of of the section $s$ of $mK_X$.  
\smallskip

\noindent Given the explicit expression of $a_1$ and $b_1$, the $L^2$ hypothesis in Theorem \ref{007'} are easy to check. ``For punishment"\footnote{\emph{Complex Manifolds}, K. Kodaira and J. Morrow}, we will give the details for $b_1$ and its derivative.

With respect to the local coordinates in \eqref{eq34}, we have 
\begin{equation}\label{eq45}
b_1|_U\simeq s_i\sum_\alpha \wh {dz_i^\alpha}\wedge \dbar v_i^\alpha
\end{equation}
and this is clearly $L^2$ with respect to the weight in \eqref{eq44}. The symbol
$\simeq$ is \eqref{eq45} means that the restriction of $b_1$ to $U$ equals the RHS
with respect to the coordinates $(z_i)$.

Moreover, the fact that $D'_X(b_1)$ is equally $L^2$ boils down to the convergence of the 
integral
\begin{equation}\label{eq46}
\int_{U_i}\frac{|df|^2}{|f|^{2\frac{m-1}{m}}}d\lambda
\end{equation}
where $f$ is a holomorphic function defined in the open set of $\CC^n$ containing $U$.
This in turn is quickly verified by a change of variables formula.
\smallskip

\noindent In conclusion, we have 
\begin{equation}\label{eq47}
\lambda_1= \dbar \gamma_1
\end{equation}
on $X$.
\smallskip

\noindent Now we construct the $2$-jet extension as in the proof of Theorem \ref{MT} 
Let $\Gamma_1$ be a smooth extension of $\gamma_1$ as in the proof of Theorem \ref{MT}, so that we have $\frac{\Gamma_1}{dt}\Big|_X= \gamma_1$ and
\begin{equation}\label{eq48}
\Lambda_0= \dbar \Gamma_1+ t\Lambda_1
\end{equation}
on the total space $\cX$, where $\Lambda_1$ is smooth forms. Let $s_1:= s_0- t\Gamma_1$
\smallskip

On the other hand, let $\displaystyle \big(\Omega_i, (t,z_i)\big)_{i\in I}$ be an arbitrary coordinate system on $\cX$. This induces automatically a trivialization of the relative canonical bundle, so our section $s$ corresponds to a family of holomorphic functions $\displaystyle (f_i)_{i\in I}$ such that we have
\begin{equation}\label{july1}
f_i(z_i)= \xi_{ij}(t, z_j)f_j(z_j)+ t\Lambda_{ij}(t,z_j)
\end{equation}
on overlapping subsets $\Omega_i\cap \Omega_j$. In \eqref{july1} we denote by
$\xi_{ij}$ the transition functions corresponding to the pluricanonical bundle
$mK_{\cX/\DD}$ and $\Lambda_{ij}$ are holomorphic on $\Omega_i\cap \Omega_j$.
By taking the derivative with respect to $t$ in \eqref{july1} and restricting to
the central fiber, we see that
\begin{equation}\label{july2}
\int_{U_{ij}}\frac{|\Lambda_{ij}(0, z_j)|^2}{|f_j(z_j)|^{2\frac{m-1}{m}}}d\lambda(z_j)< \infty
\end{equation}
where $U_{ij}= \Omega_i\cap \Omega_j\cap X$.
\smallskip

\noindent The global interpretation of this is that we can construct a $\cC^\infty$ extension 
$s_0$ of $s$ such that with respect to the coordinates fixed above we have
\begin{equation}\label{july3}
s_0|_{\Omega_i}= \left(f_i(z_i)+ t \rho_i(t, z_i)\right)(dt\wedge dz_i)^{\otimes m}
\end{equation}
where $\rho_i$ is smooth and its restriction to the central fiber verifies the integrability condition \eqref{july2}.
\smallskip

\noindent In conclusion, we can always assume that the extension $s_1$ obtained in
\eqref{eq48} can be locally written as
\begin{equation}\label{july4}
s_1|_{\Omega_i}\simeq \left(f_i(z_i)+ t f_{i1}(z_i)\right)(dt\wedge dz_i)^{\otimes m}
\end{equation}
modulo $t^2$,
where $f_{i1}$ is holomorphic and $\displaystyle \int_{U_{i}}\frac{|f_{i1}(z_i)|^2}{|f_i(z_i)|^{2\frac{m-1}{m}}}d\lambda(z_i)< \infty$.
\end{proof}

\begin{remark}\label{rk1}
This first step in the extension of the section $s$ is somehow misleading, i.e. too simple in some sense. Some of the real difficulties one has to deal with are appearing during the 
extension to the second infinitesimal neighborhood, see the subsection \ref{2nd} below.
\end{remark}

%%%%%%%%%%%%%%%%%%%%%%%%%%%%%%%%%%%%%%%%%%%%%%%%%%%%%%%%%%%%%%%%%%%%%%%%%%%%%%%%%%%%%%%%%%%%%%%%%%%%%%%%%%%%%%%%%%%%%%%%%%%%%%%%%%%%%%%%%%%%%%%%%%
\subsection{Extension of sections whose zero set is non-singular}\label{smoothdiv} In order to extend our section $s$ to the first infinitesimal neighborhood, we have used the Lie derivative 
with respect to an arbitrary vector field $\Xi$. For higher order extension, this does not seem to be possible because of the singularities of the operator $D'_\cX$. In the case of a section $s$ whose zero set $(s=0)\subset X$ is non-singular (treated in \cite{Lev83}), this is done as follows.
\smallskip

\noindent We assume that the extension $s_k$ to the $k^{\rm th}$ infinitesimal neighborhood has already been constructed. 
For each trivializing open set $\Omega_i\subset \cX$ together with fixed coordinates functions
$\displaystyle (z_i^{\alpha})_{\alpha=1,\dots, n}$ we denote by 
\begin{equation}\label{ideal3}
f_i
\end{equation}
the holomorphic function corresponding to the $k^{\rm th}$-jet of $s_k$. The set $(f_i= 0)\subset \Omega_i$ is 
non-singular and transversal to the central fiber $X$. Then we define
a new set of local coordinates 
\begin{equation}\label{ideal4}
(t, w_i^1,\dots, w_i^n)
\end{equation}
on $\Omega_i$ such that $w_i^1= f_i$ --here we use the fact that at each point of
$\Omega_i$ we can find an index $\alpha$ such that the form
\begin{equation}\label{ideal9}
df_i\wedge dt\wedge  \widehat{dz_i^\alpha}\neq 0
\end{equation}
is non-vanishing at the said point. Strictly speaking we have shrink eventually the 
set $\Omega_i$; however, given that the map $p:\cX\to \DD$ is proper we 
can assume that the coordinates \eqref{ideal4} are defined on  $\Omega_i$ itself. 
\smallskip

\noindent Therefore the equality
\begin{equation}\label{ideal5}
w_i^1= g_{i j}(t, w_j)\cdot w_j^1+ t^{k+1}\tau_{ij}(t, w_j)
\end{equation}
is valid on the overlapping subsets $\Omega_i\cap \Omega_j$, where $(g_{ij})$ are the 
transition functions for the bundle $K_\cX+ L$.
\smallskip

\noindent We introduce next the following vector field 
\begin{equation}\label{ideal6}
\Xi_k:= \sum \theta_i\frac{\partial}{\partial t}\Big|_{\Omega_i}
\end{equation}
corresponding to the covering $\big(\Omega_i, (t, w_i)\big)$. 

\noindent By the transition relation \eqref{ideal5}, we have
\begin{equation}\label{ideal7}
\Xi_k|_{\Omega_i}= \frac{\partial}{\partial t}\Big|_{\Omega_i}+ (a_i w_i^1+ b_i t^{k})\frac{\partial}{\partial w_i^1}+ 
\sum _{\alpha\geq 2}v_i^\alpha\frac{\partial}{\partial w_i^\alpha}
\end{equation} 
where $a_i, b_i$ and $v_i^\alpha$ are smooth. We notice that the vector field 
$\Xi_k$ have the following important properties.

\begin{enumerate}
\item[(i)] Its coefficients are smooth.

\item[(ii)] The projection $dp(\Xi_k)$ equals $\displaystyle \frac{\partial}{\partial t}$.

\item[(iii)] Modulo the ideal $(t^k)$, it is tangent to the set $f_i=0$.
\end{enumerate}

\smallskip
\noindent The connection we are using on the bundle $(m_0-1)K_{\cX}$ is induced by
the $\mathcal C^\infty$ section $s_k$. This means that with respect to our local coordinates in \eqref{ideal4} we have
\begin{equation}\label{ideal10}
\wt f_i(t, w)= w_i^1h_i(t, w_i)+ t^{k+1}g_i(t, w_i)
\end{equation}  
for some smooth functions $g_i, h_i$. Here $\wt f_i(t, w)$ is the local expression of the section $s_k$.

\medskip

\noindent Given these considerations, we infer the following.

\begin{lemme}\label{trivca}
Let $k\geq 0$ be a positive integer, and assume that the extension $s_k$ such that
\begin{equation}\label{ideal100}
\dbar s_k= t^{k+1}\Lambda_k
\end{equation}has been already constructed. We use the vector field $\Xi_k$ introduced in \eqref{ideal6} and consider the Lie derivative $\displaystyle\Lie_{\Xi_k}$ induced by it. Lemma \ref{linalg} provides us with forms $\alpha_k$ and $\beta_k$ such that
\begin{equation}\label{ideal101}
\Lambda_k= \dbar \alpha_k+ D'_\cX\big(dt\wedge \beta_k)
\end{equation} 
modulo the ideal generated by $t$. Then the restrictions $\alpha_k|_{X}$ and $\beta_k|_X$
are smooth.
\end{lemme}
\medskip

\noindent Before explaining the proof, we note that Lemma \ref{trivca} combined with Theorem \ref{007'} show that the restriction
\begin{equation}\label{ideal102}
\frac{\Lambda_k}{dt}\Big|_X
\end{equation}
is $\dbar$-exact. Thus we can extend $s$ one step further, given that there exist
forms $\mu_k$ and $\Lambda_{k+1}$ of type $(n+1, 0)$ and $(n+ 1, 1)$ respectively,
such that 
\begin{equation}\label{ideal103}
\Lambda_k= \dbar \mu_k+ t\Lambda_{k+1}.
\end{equation}
By combining \eqref{ideal100} and \eqref{ideal103} we infer the existence of 
$s_{k+1}$ such that 
\begin{equation}\label{ideal104}
\dbar s_{k+1}= t^{k+2}\Lambda_{k+1}.
\end{equation}
We can repeat this procedure inductively, showing that the section $s$ admits an extension to 
the infinitesimal neighborhood of an arbitrary order.
The formal arguments in \cite{Lev83} are implying that $s$ extends to a topological neighborhood of $X$ in $\cX$. 
\smallskip

\noindent We turn now to the proof of Lemma \ref{trivca}.

\begin{proof}
Considering the preparation we have done in the previous sections, the arguments which follow should be clear: we will proceed by induction, by using the fact that for each $k$ we have the relations \eqref{eq55}.

To this end, it would be helpful to remark that we have the following
while computing the obstruction to extend the section modulo $t^{k+2}$, we have to deal with quantities as
\begin{equation}\label{ideal22}
\Lie_{\Xi_k} s_k, \qquad D' _\cX \bigg( dt\wedge \Big( \Xi  \rfloor D' _\cX \big(\Xi \rfloor (\dbar \Xi\rfloor s_k)\big) \Big)\bigg)
\end{equation}
and their iterations. Locally we can write $s_k= f\sigma$, where $f$ is the function given by the expression like 
\eqref{ideal10}, and $\sigma$ is a top form with values in $(m-1)K_\cX$. We have 
\begin{equation}\label{ideal23}
\Lie_{\Xi_k} s_k= \Xi(f)\sigma+ f \Lie_{\Xi_k} \sigma
\end{equation}
in which the second term has a -potentially- singular component 
\begin{equation}\label{ideal24}
f \Lie_{\Xi_k} \sigma \equiv -\frac{m-1}{m}f\cdot \frac{df}{f}\wedge \left(\Xi_k\rfloor \sigma\right)= -\frac{m-1}{m}\Xi_k(f)\sigma.
\end{equation}
\smallskip

\noindent The point here is that the vector field $\Xi_k$ was constructed in such a way 
that $\Xi_k(f)$ is a \emph{multiple} of $f$, plus some power of $t$. 
More precisely, we have
\begin{equation}\label{ideal29}
\Xi_k\cdot (f, t^{k+1})\subset (f, t^{k})
\end{equation} 
where we remark that the ideal $(f, t^{k+1})$ is in fact globally defined on $\cX$.

\noindent Therefore we can 
take as many times the Lie derivative as the power of $t$ allows, the result will still be of the same type. We show next that the same is true for the second term in \eqref{ideal22}. 

Indeed we have 
\begin{equation}\label{ideal25}
D' _\cX \big(\Xi_k \rfloor (\dbar \Xi_k\rfloor s_k)\big)= \partial f\wedge \Xi_k \rfloor (\dbar \Xi_k\rfloor \sigma)+
f D' _\cX \big(\Xi_k \rfloor (\dbar \Xi_k\rfloor \sigma)\big)
\end{equation}
and a further contraction with $\Xi$ gives
\begin{equation}\label{ideal26}
\Xi_k\rfloor \left(\partial f\wedge \Xi_k \rfloor (\dbar \Xi_k\rfloor \sigma)\right)= \Xi_k(f)\cdot \Xi_k\rfloor (\dbar \Xi_k\rfloor \sigma)
\end{equation}
for the first term in \eqref{ideal25}. The singular part of the second one is 
\begin{equation}\label{ideal27}
f\cdot \frac{df}{f}\wedge \Xi_k \rfloor (\dbar \Xi_k\rfloor \sigma)
\end{equation}
and when we contract it with $\Xi_k$ the result will be the RHS of \eqref{ideal26}. 
\medskip

\noindent In conclusion, after the first derivative of the relation \eqref{eq339} the forms 
$u_2$ and $v_2$ of \eqref{eq55} have their coefficients in the ideal $(f, t^k)$. 
We can therefore iterate this procedure, and obtain the conclusion.
\end{proof}

\smallskip

\begin{remark} We note that at each step $k$ we choose a vector field $\Xi_k$ adapted to the corresponding extension $s_k$. 
\end{remark}
\smallskip

\begin{remark} If the zero set $(s=0)$ of our initial section is singular, then one can still 
construct a vector field adapted to it as in the proof just finished, 
but the difference is that the new $\Xi$ will be singular along $s=0, ds= 0$.
\end{remark}
%%%%%%%%%%%%%%%%%%%%%%%%%%%%%%%%%%%%%%%%%%%%%%%%%%%%%%%%%%%%%%%%%%%%%%%%%%%%%%%%%%%%%%%%%%%%%%%%%%%%%%%%%%%%%%%%%%%%%%%%%%%%%%%%%%%%%%%%%%
\subsection{A sufficient condition}\label{sufcond} In this subsection we assume that the divisor $\Div(s)$ corresponding to the section $s\in H^0(X, mK_X)$ has the following property.
\begin{hyp}\label{cond1}
There exists a divisor $\Sigma$ on $\cX$ such that we have. 
\begin{enumerate}
\smallskip

\item[(i)] The divisor $\Sigma+ X$ is snc.
\smallskip

\item[(ii)] The support of $\Div(s)$ is contained in the restriction $\displaystyle \Sigma|_X$.
\end{enumerate}
\end{hyp}
\medskip

\noindent Then we show that the following holds true.

\begin{thm}\label{hypo1} Let $s\in H^0(X, mK_X)$ be a pluricanonical section such that Hypothesis \ref{cond1} is satisfied. Then $s$ admits a holomorphic extension to $\cX$. 
\end{thm}

\begin{proof} The role of the two conditions (i) and (ii) above is to furnish a special coordinate system on $\cX$ which we will use during all the proof. 

\noindent As consequence of these requirements we have a finite covering $\displaystyle \big(\Omega, (t, z_i)\big)_{i\in I}$ of an open subset of the central fiber $X$ of $p$ such that for each $i\in I$ we have $\Sigma\cap \Omega_i= (z_i^1\dots z_i^r= 0)$ together with 
the transition functions 
\begin{equation}\label{suff1}
z_i^\gamma= g_{ij}^\alpha(t, z_j)z_j^\gamma,\qquad z_i^\beta= \phi_{ij}^\beta(t, z_j)
\end{equation}
for each $\gamma= 1,\dots, r$ and $\beta= r+1,\dots, n$. We denote the resulting function by $\Phi$.
\smallskip

\noindent These coordinates are inducing a trivialization of the canonical bundle of $K_\cX$ with respect to which our section $s$ corresponds locally to holomorphic functions 
\begin{equation}\label{suff2}
f_{i0}(z_i)= g_i(z_i)\prod_{q=1}^r(z_i^q)^{m_q}
\end{equation}
where $g_i$ is a nowhere vanishing holomorphic function on $\Omega_i\cap X$ and $m_q\geq 0$.
\medskip

\noindent We formulate the next statement.
\begin{claim}\label{blb}
For every order $k\geq 0$ there exists a section 
$s_k\in H^0(\cX, \cF_k)$ --with $\Lie= (m-1)K_\cX$-- such that $\displaystyle s_k|_X= s$ and such that it can be written locally as follows
\begin{equation}\label{suff3}
f_i(t, z_i)= \sum_{\alpha=0}^k \frac{t^\alpha}{\alpha!}f_{i\alpha}(z_i)
\end{equation}
where $f_{i\alpha}$ are holomorphic, such that for every positive real $\ep> 0$ and for every $\alpha=0,\dots, k$ we have 
$\displaystyle \int_{X\cap \Omega_i}\frac{|f_{i\alpha}|^2}{|f_{i0}|^{2(1-\ep)\frac{m-1}{m}}}< \infty$.
\end{claim}
\medskip

\noindent We prove next Claim \ref{blb} by induction on $k$. For $k=0$ this is simply the hypothesis, so let us assume that a section $s_k$ as in \eqref{suff3} exists. We are constructing next $s_{k+1}$. If we denote by $\xi_{ij}$ the transition functions of the bundle $mK_\cX$, then we have
\begin{equation}\label{suff4}
f_i(t, z_i)= \xi_{ij}(t, z_j)\left(f_j(t, z_j)+ t^{k+1}\Lambda_{ij}(t, z_j)\right)
\end{equation}
on the overlapping sets $\Omega_i\cap \Omega_j$. By the equalities \eqref{suff1} it follows that 
\begin{equation}\label{suff5}
\frac{\partial^l f_{i\alpha}\big(\Phi(t, z_j)\big)}{\partial t^l}\Big|_{t=0}
\end{equation}
is $L^2$ integrable with respect to the measure $\displaystyle \frac{d\lambda(z_i)}{|f_{i0}|^{2(1-\ep)\frac{m-1}{m}}}$ for any positive integer $l$ and for every positive $\ep> 0$. Indeed, given \eqref{suff1} if some monomial, say, $\displaystyle (z_i^1)^{\delta_1}$ divides $f_{i\alpha}$, the same is true for any derivative with respect to $t$ as in \eqref{suff5} since the effect of the transition functions is to replace this monomial by  $\displaystyle (g_{ij}^1(t, z_j)z_j^1)^{\delta_1}$. Thus the $L^2$ condition that we are imposing to each component $f_{i\alpha}$ is preserved.

\noindent The conclusion is that we have 
\begin{equation}\label{suff6}
 \int_{X\cap \Omega_{ij}}\frac{|\Lambda_{ij}(0, z_j)|^2}{|f_{j0}(z_j)|^{2(1-\ep)\frac{m-1}{m}}}d\lambda(z_j)< \infty
\end{equation}
for every positive $\ep> 0$ which means precisely that the $(n, 1)$ form $(m-1)K_X$-valued form $\lambda$ associated to the cocycle 
$(\Lambda_{ij}|_X)_{i,j}$ is verifies the hypothesis of Theorem \ref{MT}. 
\smallskip

\noindent The proof of Theorem \ref{MT} shows that $\displaystyle \frac{\lambda}{\prod s_{Y_i}^{\nu_q}}$ is $\dbar$-exact, where $Y_i$ are the components of $\Sigma|_X$ and 
$\displaystyle \nu_q:= \left\lfloor \frac{m-1}{m} m_q\right\rfloor$
if $\displaystyle \frac{m-1}{m} m_q\not\in \Z$ and $\displaystyle \nu_q:= \frac{m-1}{m}m_q- 1$ if this number is an integer.
\smallskip

\noindent This can be re-interpreted in terms of cocycles as follows: there exists holomorphic $L^2$ functions $f_{j k+1}(z_j)$ such that we have 
\begin{equation}\label{suff7}
\Lambda_{ij}(0, z_j)= f_{j k+1}(z_j)- \xi_{ji}(0, z_i)f_{i k+1}(z_i),
\end{equation}
such that the holomorphic functions $f_{j k+1}$ verify the $L^2$ requirement in the Claim. When combined with \eqref{suff4}, this completes the proof Theorem \ref{hypo1}.
\end{proof}
\medskip

\begin{remark} In general (i.e. without the special co-ordinate system provided by the 
hypothesis at the beginning of this subsection) a section of $mK_{\cX}$ which is holomorphic mod 
$t^{2}$ is described by the  
co-cycle relation
\begin{equation}\label{suff4bis}
f_i(t, z_i)= \xi_{ij}(t, z_j)\left(f_j(t, z_j)+ t^{2}\Lambda_{ij}(t, z_j)\right)
\end{equation}
on the overlapping sets $\Omega_i\cap \Omega_j$. Moreover, we can assume that the equality
\begin{equation}
f_i(t, z_i)= f_{i}(z_i)+ tf_{i 1}(z_i)
\end{equation}
holds, with $f_{i \alpha}$ belonging to the space $L^{2}$ with respect to the metric $\displaystyle 
\frac{m-1}{m}\log|s|^2$ on the central fiber, cf. Remark \ref{rk1}. 
\smallskip

\noindent It follows that the function $z_j\to \Lambda_{ij}(0, z_j)$ in \eqref{suff4bis} belongs to the ideal generated by the functions $f_j$ (i.e. our initial section $s$) and their partial derivatives up to order $2$. 
\end{remark}
\smallskip

\noindent In this context, would be good to know the answer to the following question.

\begin{quest}
Does it follows that the resulting $(n, 1)$-form $\lambda$ can be written as
\begin{equation}\label{megasuff1}
\lambda= D'u+ v
\end{equation}
where $v$ is $L^2$ and $u$ is obtained by successive integration, both being $\dbar$-closed? It is certainly the case locally. A positive answer to this question would certainly be a good news: by the global Lie derivative argument 
we know that $\lambda$ belongs to the image of $D' +\dbar$ so writing 
\begin{equation}\label{megasuff2}
v= D'(\alpha- u)+ \dbar \beta
\end{equation}
we will be able to show that $v$ is $\dbar$-exact. 
\end{quest}

\medskip
%%%%%%%%%%%%%%%%%%%%%%%%%%%%%%%%%%%%%%%%%%%%%%%%%%%%%%%%%%%%%%%%%%%%%%%%%%%%%%%%%%%%%%%%%%%%%%%%%%%%%%%%%%%%%%%%%%%%%%%%%%%%%%%%%%%%%%%%%%%
\subsection{Extension to the second infinitesimal neighborhood}\label{2nd} We have already mentioned that the techniques we are developing in this paper are allowing us -in some particular cases- to extend the section $s$ to the second infinitesimal neighborhood of the central fiber. We present the arguments in this section.
\smallskip

\noindent Thanks to the subsection \ref{firstorderex}, we can always find a $1$-order extension $s_2$, i.e.
$$s_2 |_X =s , \qquad \dbar s_2 =t^2 \Lambda_2 .$$
The main observation is that if we write 
\begin{equation}\label{part1}
\Lambda_2= \dbar \alpha_2+ D'_\cX\big(dt\wedge \rho_2),
\end{equation}
as in Lemma \ref{trivca}, then the restriction $\displaystyle \rho_2|_X$ is 
\emph{automatically} in $L^2$ under the assumption of  Theorem \ref{order2}. 
Here the properties \eqref{stupid_cond1} of the vector field $\Xi$ used in order to define the Lie derivative. This is the content of our next statement.

\begin{lemme}Under the hypothesis of Theorem \ref{order2}, we have 
\begin{equation}\label{part2}
\int_X|\rho_2|^2_{\omega_E}e^{-\varphi_L}dV_{\omega_E}< \infty.
\end{equation}
\end{lemme}

\begin{proof} This will be done by an explicit evaluation of $\rho_2$. Our first claim is 
that it is enough to show the convergence of the integrals
\begin{equation}\label{part3}
\int_{(\CC^n, 0)}|f\wh{dz^\alpha}\wedge \dbar v|^2_{\omega_E}\frac{dV_{\omega_E}}{|f|^{2\frac{m-1}{m}}} 
\end{equation}
and 
\begin{equation}\label{part4}
 \int_{(\CC^n, 0)}|\partial f\wedge \wh{dz^{\alpha, \beta_k}}\wedge\dbar v|^2_{\omega_E}\frac{dV_{\omega_E}}{|f|^{2\frac{m-1}{m}}},
\end{equation}
as well as
\begin{equation}\label{part4t}
 \int_{(\CC^n, 0)}|\partial_t f|^2|\dbar \Xi|^2_{\omega_E}\frac{d\lambda}{|f|^{2\frac{m-1}{m}}}.
\end{equation}
This is a consequence of formula 
\eqref{eq55}, the computations are as follows. We have
\begin{equation}\label{calc3}
\dbar\Xi\rfloor \Lie_\Xi(s)\simeq \frac{1}{m}\Xi(f)\cdot \dbar\Xi\rfloor (dt\wedge dz) 
\end{equation}
modulo a term divisible by $f$. On the other hand, we have 
\begin{equation}\label{calc4}
D'\big(\Xi\rfloor(\dbar\Xi\rfloor s)\big)\simeq \frac{1}{m} \partial f\wedge \Big(\Xi\rfloor \big(\dbar\Xi\rfloor(dt\wedge dz)\big)\Big)
\end{equation}
again modulo a multiple of $f$. A further contraction with $\Xi$ gives 
\begin{equation}\label{calc5}
\Xi\rfloor \Big(D'\big(\Xi\rfloor(\dbar\Xi\rfloor s)\big)\Big)\simeq 
\frac{1}{m} \Xi(f)\cdot \Xi\rfloor \big(\dbar\Xi\rfloor(dt\wedge dz)\big).
\end{equation}
Finally, we apply $dt\wedge \cdot$ in \eqref{calc5}, and what we get is the same as 
the RHS of \eqref{calc3}. Now given the formula which computes 
$dt\wedge \rho_2$ our claim follows.
\smallskip

Coming back to the integrals above, for \eqref{part3} things are clear because the volume of $(X, \omega_E)$ is finite.
Modulo the blow-up map $\pi$ in Section \ref{red}, the integral \eqref{part4} reduces to the evaluation of the next quantity
\begin{equation}\label{part4mm}
 \sum_{\beta_k=1}^p\int_{(\CC^n, 0)} \prod_{\alpha=1}^p |z_\alpha|^{2\delta_\alpha}\left|\frac{\wh {dz_\beta}}{z_\beta}\right|^2_{ \omega_E} dV_{\omega_E}.       
 \end{equation}
 where $\delta_\alpha> 0$ are positive rational numbers. The convergence of \eqref{part4mm}
 follows, given that we know the singularities of the metric $\omega_E$.

 \noindent For the expression \eqref{part4t} we use the fact that after the first order extension cf. Section \ref{firstorderex},  we have
\begin{equation}\label{part200}
 \int_{(\CC^n, 0)}|\partial_t f|^2\frac{d\lambda}{|f|^{2\frac{m-1}{m}}}< \infty.
\end{equation} 
But then it follows that we also have
\begin{equation}\label{part201}
 \int_{(\CC^n, 0)}|\partial_t f|^2\frac{d\lambda}{|f|^{2\frac{m-1}{m}+ \ep_0}}< \infty
\end{equation}
for some positive real $\ep_0> 0$. This is enough to absorb the term arising from $|\dbar \Xi|^2_{\omega_E}$, because of the hypothesis \eqref{stupid_cond1}.
 \end{proof}
 \medskip

 \noindent In conclusion, we find ourselves in the following situation i.e. the setting of Theorem \ref{L2I}: let $(L, h_L)$ be a holomorphic line bundle on $X$ and the possible singular metric $h_L$ is of analytic singularities. Let $Z$ be the singular locus of $h_L$ and let $\omega_E$ be a Poincaré type metric on $X$ with poles along the $Z$. We have a $L$-valued $L^2$ form $\rho$ of type $(n-1,1)$ on $X$, such that 
 \begin{equation}
 \tau:= D'_X\rho
 \end{equation}
 is $\dbar$ closed. If $\tau$ would be $L^2$, then we can apply directly 
 Theorem \ref{007'} and conclude that $\tau$ is $\dbar$-exact. However, we do not poses
 this information, and we will follow a different path.
 \smallskip
 
 \noindent We define the linear form $T_\rho$ by the formula
\begin{equation}
T_\rho (\psi):= \int_X\langle \rho, (D' _X)^\star \psi\rangle_{\omega_E}e^{-\varphi_L}
dV_{\omega_E}
\end{equation} 
where $\psi$ is a $L$-valued smooth form of $(n, 1)$-type with compact support in $X\setminus Z$.

\noindent The current $T_\rho$ has the following properties.
\begin{lemme}\label{ort} Let $\psi$ be a test form as above, and consider the decomposition 
\begin{equation}\label{part6}
\psi= \xi_1+ \xi_2
\end{equation}
according to $\Ker(\dbar)$ and its orthogonal. Then $T_\rho(\xi_2)= 0$.
\end{lemme}

\begin{proof} Note that $T_\rho(\xi_2)$ is indeed well-defined, thanks to Friedrichs lemma
(cf. e.g. \cite{bookJP}): the $L^2$ form $\xi_2$ belongs to the domain of $\dbar$, hence it is a limit in graph norm of smooth forms $(\xi_{2,k})$ with compact support. It follows that we have 
\begin{equation}\label{part7}
(D' _X)^\star (\xi_{2,k})\to (D' _X)^\star (\xi_{2})
\end{equation}
in $L^2$ as $k\to\infty$. This is a consequence of Bochner formula
\begin{equation}\label{part8}
\int_X |(D' _X)^\star (\xi_{2,k}- \xi_{2})|^2dV\leq \int_X |\dbar (\xi_{2,k}- \xi_{2})|^2dV+
\int_X |\dbar^\star (\xi_{2,k}- \xi_{2})|^2dV
\end{equation}
(here semi-positive curvature is sufficient) and by Friedrichs lemma the RHS of \eqref{part8} tends to zero as $k\to\infty$.
\smallskip

\noindent It follows that $T_\rho(\xi_{2})$ is well-defined and moreover we have 
\begin{equation}\label{part12}
T_\rho(\xi_{2})= \int_X\langle \rho, (D' _X)^\star \xi_{2}\rangle_{\omega_E}e^{-\varphi_L}.
\end{equation}
Thanks to the $L^2$ hypothesis on $\rho$, we have 
\begin{equation}\label{part13}
\int_X\langle \rho, (D' _X)^\star \xi_{2}\rangle_{\omega_E}e^{-\varphi_L}=
\lim_{\ep\to 0}\int_X\langle \rho, (D' _X)^\star (\mu_\ep\xi_{2})\rangle_{\omega_E}e^{-\varphi_L}
\end{equation}
where $\displaystyle (\mu_\ep)_{(\ep> 0)}$ is the family of cut-off functions 
adapted to Poincar\'e metric cf. \cite[Lemma 2.1]{CP23}

\noindent Indeed, \eqref{part13} follows since we have 
\begin{equation}\label{part14}
D'^\star (\mu_\ep\xi_{2})= \mu_\ep D'^\star (\xi_{2})+ {\partial \mu_\ep}\rfloor \xi_2
\end{equation}
and again by the $L^2$ condition, we have
\begin{equation}
\lim_{\ep\to 0}\int_X\langle \rho, D'^\star (\mu_\ep\xi_{2})\rangle_{\omega_E}e^{-\varphi_L}=
\lim_{\ep\to 0}\int_X\langle \rho, \mu_\ep D'^\star (\xi_{2})\rangle_{\omega_E}e^{-\varphi_L}.
\end{equation}

\noindent Therefore, we infer that 
\begin{equation}
\int_X\langle \rho, D'^\star \xi_{2}\rangle_{\omega_E}e^{-\varphi_L}= 
\lim_{\ep\to 0}\int_X\langle \mu_\ep D'\rho, \xi_{2}\rangle_{\omega_E}e^{-\varphi_L}.
\end{equation}
On the other hand, $\xi_2= \dbar^\star \lambda$ for some $L^2$ form $\lambda$
which can be assumed to be $\dbar$-closed
and then we write 
\begin{equation}\label{part171}
\int_X\langle \mu_\ep D'\rho, \xi_{2}\rangle_{\omega_E}e^{-\varphi_L}=
\int_X\langle \dbar \mu_\ep \wedge D'\rho, \lambda\rangle_{\omega_E}e^{-\varphi_L}
\end{equation}
because $D'\rho$ is $\dbar$-closed. The RHS term in \eqref{part171} is equal to 
\begin{equation}\label{part181}
\int_X\langle D'(\dbar \mu_\ep \wedge \rho), \lambda\rangle_{\omega_E}e^{-\varphi_L}
+ \int_X\langle \ddbar \mu_\ep \wedge \rho, \lambda\rangle_{\omega_E}e^{-\varphi_L}
\end{equation}
up to a sign. The second term of \eqref{part181} rends to zero as $\ep\to 0$, and so does 
the first one, because by Bochner formula we have 
\begin{equation}\label{part191}
\int_X|D'^{\star}(\lambda)|^2dV\leq \int_X|\dbar^\star\lambda|^2dV= \int_X|\xi_2|^2dV 
\end{equation}
(because we assume that $\dbar\lambda=0$), and the RHS of \eqref{part191}
is convergent. Note again that here the semi-positivity of the curvature is enough, given that $\lambda$ is of type $(n,2)$. 
\smallskip

\noindent In conclusion, the $T_\rho(\xi_{2})=0$.
 \end{proof}
 \smallskip
 
 \noindent Now we are ready to prove Theorem \ref{order2} and Theorem \ref{L2I}.
 
 \begin{proof}[Proof of Theorem \ref{L2I}]
Let $\psi$ be an $L$-valued smooth form of $(n, 1)$-type with compact support in $X\setminus Z$. Consider the Hodge decomposition \eqref{part6}; then we have
\begin{equation}
T_{v}(\psi)= T_v (\xi_1)
\end{equation} 
by Lemma \ref{ort}. Then it follows by Cauchy-Schwarz inequality combined with Bochner formula and 
the usual $L^2$ theory that there exists a $u$ such that
\begin{equation}
T_v (\psi)= \int_X\langle u, \dbar^\star \psi\rangle e^{-\varphi_L} 
\end{equation}
in other words we infer that 
\begin{equation}
D' _X (v)=  \dbar u  .
\end{equation}
 \end{proof}

  \begin{proof}[Proof of Theorem \ref{order2}]
 Our aim is to prove that the $\Lambda_2  |_X$ in \eqref{part1} is $\dbar$-exact,
which is equivalent to prove that the $\dbar$-closed form $D' _X ( \rho_2 |_X)$ is $\dbar$-exact.
It is a direct consequence of  Lemma \ref{ort} and Theorem \ref{L2I}.
  \end{proof}

\medskip

\noindent We prove next Theorem \ref{L2II}. Actually the motivation for this
result is that the form $\displaystyle \rho_2|_X$ is automatically $L^2$ \emph{if} we are using a non-singular metric on $X$.

\begin{proof}[Proof of Theorem \ref{L2II}]
  The main technical difficulty here is that the metric $\omega$ on $X\setminus Z$ is not complete. Usually this is bypassed by using a sequence of complete metrics
\begin{equation}\label{part15}
\omega_\ep:= \omega+ \ep \omega_Z
\end{equation}
and invoke the usual arguments in $L^2$ theory. This works perfectly for forms of type $(n, q)$ (because in this case, a monotonicity argument can be used) but in our case the form $v$ is of type $(n-1,1)$, hence in the absence of hypothesis (3), it is not necessarily $L^2$ with respect to $\omega_\delta$ above. We proceed as follows.
\smallskip

\noindent We recall that the weight $\varphi_L$ of the metric $h_L$ is assumed to have log poles along $Z$.
Let $\xi$ be an $L$-valued $(n, 1)$-form whose support is contained in $X\setminus Z$. We have to evaluate the quantity
\begin{equation}\label{part16}
\int_X\langle D'v, \xi\rangle e^{-\varphi_L}dV_\omega.
\end{equation}
%and our first observation is that this equals
%\begin{equation}\label{part17}
%\lim_{\ep\to 0}\int_X\langle D'v, \xi\rangle e^{-\varphi_{L, \ep}}dV_\omega.
%\end{equation}
The form $\xi$ can be written as
\begin{equation}\label{part18}
\xi= \xi_{1}+ \xi_{2}
\end{equation}  
according to $\Ker (\dbar)$ and its orthogonal. Since $D'v$ is $L^2$ and $\dbar$-closed, we have
\begin{equation}\label{part19}
\int_X\langle D'v, \xi_{2}\rangle e^{-\varphi_{L}}dV_\omega= 0 .
\end{equation}
So \eqref{part16} equals the expression
\begin{equation}\label{part20}
\int_X\langle D'v, \xi_{1}\rangle e^{-\varphi_{L}}dV_\omega.
\end{equation}

\bigskip

\noindent Note that we have 
\begin{equation}\label{part21}
\int_X\langle D'v, \xi_{1}\rangle e^{-\varphi_{L}}dV_\omega=
\lim_{\ep\to 0}\int_X\langle D'v, \xi_{1, \ep}\rangle e^{-\varphi_{L}}dV_{\omega_\ep}.
\end{equation}
In \eqref{part21}, the notations are as follows. The metric $\omega_\ep$ was introduced in 
\eqref{part15}, and $\xi_{1, \ep}$ is the orthogonal projection of $\xi$ on the $\Ker(\dbar)$ 
with respect to $(X\setminus Z, \omega_\ep)$ and $(L, h_L)$, respectively.
The equality \eqref{part21} follows from the fact that $\xi_{1, \ep}\to \xi_1$ uniformly on the compact sets of $X\setminus Z$.
\smallskip

\noindent It is at this point that we have to use the hypothesis (3): as a consequence of it, we can write
\begin{equation}\label{part22}
\int_X\langle D'v, \xi_{1, \ep}\rangle e^{-\varphi_{L}}dV_{\omega_\ep}= \int_X\langle v, D'^\star(\xi_{1, \ep})\rangle e^{-\varphi_{L}}dV_{\omega_\ep}.
\end{equation}
Moreover we have 
\begin{equation}\label{part233}
\lim_{\ep\to 0}\int_X|v|_{\omega_\ep}^2dV_{\omega_\ep}= \int_X|v|_{\omega}^2dV_{\omega}
\end{equation}
because if we write in coordinates 
\begin{equation}\label{part24}
v= \sum v_{p\ol q}\wh{dz_p}\wedge \ol{dz_q}, 
\end{equation}
then we have 
\begin{equation}\label{part25}
|v|_{\omega_\ep}^2dV_{\omega_\ep}= \sum |v_{p\ol q}|^2\frac{1+ \ep\lambda_p}{1+ \ep\lambda_q}dz \wedge d\ol z
\end{equation}
and the trivial inequality $\displaystyle \frac{1+ \ep\lambda_p}{1+ \ep\lambda_q}< 1+ \frac{\lambda_p}{\lambda_q}$ implies that $|v|_{\omega_\ep}^2dV_{\omega_\ep} \leq 
|v|_{\omega}^2dV_{\omega} + |v|_{\omega_Z}^2dV_{\omega_Z}$.
This allows us to use dominated convergence theorem and infer \eqref{part233}. 
\smallskip

The $L^2$ norm of the form $D'^\star(\xi_{1, \ep})$ is smaller than 
\begin{equation}\label{part26}
\int_X|\dbar^\star(\xi_{1, \ep})|^2e^{-\varphi_{L}}dV_{\omega_\ep}= \int_X|\dbar^\star \xi|^2e^{-\varphi_{L}}dV_{\omega_\ep}
\end{equation}
and since we assume from the 
beginning that the support of the form $\xi$ is contained in $X\setminus Z$, the limit as 
$\ep\to 0$ of the RHS of \eqref{part26} is precisely 
\begin{equation}\label{part27}
\int_X|\dbar^\star\xi|^2e^{-\varphi_{L}}dV_{\omega},
\end{equation}
and we are done.
\end{proof}
\medskip

\subsection{An application of Theorem \ref{order2}} In this subsection we consider the simplest case of a pluricanonical section whose corresponding divisor is unreduced. Let $s\in H^0(X, mK_{X})$ such that
\begin{equation}\label{july5}
\Div(s)= 2Z 
\end{equation}
where $Z$ is a non-singular hypersurface of the central fiber $X$ of our family.
\medskip

\noindent We have the following result.
\begin{thm}\label{unred} Let $s\in H^0(X, mK_{X})$ be a pluricanonical section, such that \eqref{july5} is satisfied. Then there exists a section $s_2\in\cC^{\infty}(\cX, mK_{\cX})$ such that the following hold
\begin{equation}\label{july6}
s_2|_X= s, \qquad \dbar s_2= t^3\Lambda_2.
\end{equation}  
\end{thm}

\begin{proof}
  We will show that under the hypothesis \eqref{july5} we can construct the vector field $\Xi$ as requested in Theorem \ref{order2}.
\smallskip

\noindent To this end, we first use Subsection \ref{firstorderex} , so that the first order extension $s_1$ of $s$ can be written as 
\begin{equation}\label{july7}
s_1|_{\Omega}\simeq \left(z_1^2+ tz_1f_{1}(z)\right)(dt\wedge dz)^{\otimes m}
\end{equation}
modulo $t^2$, where $\Omega$ is one of the $\Omega_i$, $z_1= 0$ is the equation of the divisor $Z$ on the central fiber $X$ and $f_1$ is holomorphic.
\smallskip

\noindent We change the coordinates on $\Omega$ by taking
\begin{equation}\label{july8}
w_1:= z_1+ \frac{t}{2}f_1(z), \qquad w_i= z_i
\end{equation}
for $i=2,\dots, n$. The new local expression for the section $s_1$ reads as
\begin{equation}\label{july9}
s_1|_{\Omega}\simeq w_1^2\big(1+ tg(w)\big)(dt\wedge dw)^{\otimes m}
\end{equation}
where $g$ is holomorphic, again modulo smooth terms divisible by monomials in $t$
and $\ol t$ of degree at least 2.
\smallskip

\noindent Let $h_s$ be the Hermitian metric on $\cX\setminus (s_1=0)$ defined as follows
\begin{equation}\label{july10}
h_s:= \omega+ \frac{\sqrt{-1}D's_1\wedge \ol{D's_1}}{|s_1|^2}
\end{equation}
where $\omega$ is an arbitrary smooth metric on $\cX$ and $D'$ is the covariant derivative given by the Chern connection induced by a smooth metric on $mK_\cX$.    
\smallskip

\noindent The metric $h_s$ is singular and in general not Kähler, but its task is to allow us to define $\Xi$ as the canonical lifting of the vector field
$\displaystyle \frac{\partial}{\partial t}$ in an intrinsic, computable manner. 
Recall that with respect to the coordinates $(t, w)$ as above we have
\begin{equation}\label{july11}
\Xi|_\Omega= \frac{\partial}{\partial t}- \sum v_i\frac{\partial}{\partial w_i}
\end{equation}  
where the coefficients $\displaystyle v_i:= \sum_{\alpha} h^{\ol \alpha i}h_{t\ol\alpha}$ are
expressed by the usual formula with respect to the coefficients of $h_s$.
\smallskip

\noindent We evaluate next briefly the coefficients of $h_s$; we write
\begin{equation}\label{ibs1}
D's_1|_\Omega= 2w_1dw_1+ \O(w_1^2)+ \O(t, \ol t)
\end{equation}
where $\O(w_1^2)$ and  $\O(t, \ol t)$ denote forms of type (1,0) on $\Omega$ whose coefficients belong to the ideal $(w_1^2)$ and $(t, \ol t)$, respectively. Therefore, at points $\Omega\cap X$ we can write 
\begin{equation}\label{ibs2}
h_{t\ol 1}= g_{t\ol 1}+ \O\Big(\frac{1}{\ol w_1}\Big), \qquad h_{t\ol \alpha}= g_{t\ol \alpha}+ \O(1)
\end{equation}
for each $\alpha=2,\dots, n$, where $g_{t\ol \alpha}$ are the coefficients of the non-singular metric $\omega$ Moreover, the determinant of the matrix corresponding to $h_s$ is clearly equal to $\displaystyle \frac{\O(1)}{|w_1|^2}$, where here $\O(1)$ is smooth, 
positive and bounded away from zero. 

\noindent Thus the coefficients of the inverse matrix can be written as
\begin{equation}\label{ibs3}
h^{\ol 1 1}= |w_1|^2\big(1+ \O(1)\big),\qquad h^{\ol\alpha 1}= |w_1|^2\Big(1+ \frac{\O(1)}{\ol w_1}\Big)
\end{equation} 
for each $\alpha= 2, \dots, n$
and hence we have
\begin{equation}\label{ibs4}
v_1= \O(w_1),\qquad v_i= \O(1)
\end{equation}
for each $i=2,\dots, n$.
\medskip

\noindent  This case-by case analysis shows that the requirements of Theorem {\ref{order2}} are satisfied, since in our present case the Poincaré metric is quasi-isometric to 
\begin{equation}\label{ibs5}
\frac{\sqrt{-1}dw_1\wedge d\ol w_1}{|w_1|^2\log^2|w_1|^2}+ \sum_{i=2}^n \sqrt{-1}{dw_i\wedge d\ol w_i}.
\end{equation}
\noindent This ends the proof of our theorem. 
\end{proof}

\end{document}